\documentclass[12pt]{amsart}
\usepackage[dvips]{color}
\usepackage{amsmath}
\usepackage{amsxtra}
\usepackage{amscd}
\usepackage{amsthm}
\usepackage{amsfonts}
\usepackage{amssymb}
\usepackage{eucal}
\usepackage{epsfig}
\usepackage{graphics}
\usepackage{accents}
\usepackage{simplewick}

\usepackage{mathtools}
\usepackage{tikz}
\usepackage[all]{xy}
\usepackage{here} 
\usetikzlibrary{calc}
\numberwithin{equation}{section}
\newtheorem{thm}{Theorem}[section]
\newtheorem{prop}[thm]{Proposition}
\newtheorem{lem}[thm]{Lemma}

\newtheorem{conj}[thm]{Conjecture}

\newtheorem{dfn}[thm]{Definition}

\newcommand{\nn}{\nonumber}

\theoremstyle{definition}

\newcommand{\C}{{\mathbb C}}
\newcommand{\Z}{{\mathbb Z}}

\newcommand{\cA}{{\mathcal A}}

\newcommand{\E}{{\mathcal E}}

\newcommand{\cM}{\mathcal{M}}

\newcommand{\cR}{\mathcal{R}}

\newcommand{\gl}{\mathfrak{gl}}

\newcommand{\ssq}{\textsf{q}}

\newcommand{\id}{{\rm id}}

\newcommand{\Tr}{{\rm Tr}}

\newcommand{\g}{\mathfrak{g}}

\newcommand{\diag}{\mathop{\mathrm{diag}}}
\newcommand{\CP}{\mathbb{P}}
\newcommand{\Sp}{\mathrm{Sp}}
\newcommand{\Pole}{\mathrm{Pole}}

\newcommand{\pz}[1]{{p}_{#1}}

\newcommand{\ba}{\bar{a}}
\newcommand{\bb}{\bar{b}}
\newcommand{\bc}{\bar{c}}

\begin{document}

\begin{title}[$q$-difference opers]
{Remarks on $q$-difference opers 
\\
arising from quantum toroidal algebras}
\end{title}

\author{B. Feigin, M. Jimbo, and E. Mukhin}

\address{BF: 
National Research University Higher School of Economics,  101000, Myasnitskaya ul. 20, Moscow,  Russia, and Landau Institute for Theoretical Physics, 142432, pr. Akademika Semenova 1a, Chernogolovka, Russia;
Hebrew University of Jerusalem, Einstein Institute of Mathematics,
Givat Ram. Jerusalem, 9190401, Israel
}
\email{bfeigin@gmail.com}

\address{MJ: 
Professor Emeritus,
Rikkyo University, Toshima-ku, Tokyo 171-8501, Japan}
\email{jimbomm@rikkyo.ac.jp}

\address{EM: Department of Mathematics,
Indiana University Indianapolis,
402 N. Blackford St., LD 270, 
Indianapolis, IN 46202, USA}
\email{emukhin@iu.edu}


\begin{abstract} 
We propose a conjectural correspondence between 
the spectra of the Bethe algebra for the quantum toroidal $\mathfrak{gl}_2$
algebra on relaxed Verma modules, 
and $q$-hypergeometric opers with apparent singularities.
We introduce alongside 
the notion of apparent singularities for linear $q$-difference 
operators and discuss some of their properties.
We also touch on a generalization to $\mathfrak{gl}_n$.
\end{abstract}

\keywords{$q$-difference oper, apparent singularity, Bethe ansatz, 
quantum toroidal algebra}

\hfill  \today \bigskip

\maketitle

\section{Introduction}

Quantum toroidal algebras have large commutative subalgebras called 
the Bethe algebras, 
whose members are deformations of integrals of motion in conformal field theory.
The description of their spectra is an interesting and important problem.
Although it has been addressed for highest weight modules 
some time ago \cite{AO,FJMM1,FJMM2,FJM},  there remain many issues to be resolved. 
The present note is an attempt to treat an example of non-highest weight modules 
using the language of $q$-difference opers. 

Opers are certain gauge equivalence classes of linear systems of
differential operators associated with simple or affine Lie algebras. 
Here we are concerned only with type A where opers 
are uniquely represented by scalar differential operators, 
so we will use the terms ``oper'' and ``scalar operator'' interchangeably. 
The $q$-difference counterpart of opers originates in 
Baxter's TQ relations 
well known in integrable lattice models \cite{Ba}.
Let us recapitulate their main features in the simplest case of $\mathfrak{sl}_2$.

Fix $p\in\C^{\times}$.
The Bethe algebra of  $U_q\widehat{\mathfrak{sl}}_2$ 
is generated by transfer matrices, which are weighted traces of 
the universal R matrix $\mathcal{R}$,
\begin{align*}
T_V(x)=\mathrm{Tr}_{V}\Bigl(
\bigl(p^{-h/2}\otimes 1\bigr) \mathcal{R}(x)\Bigr)\,,
\qquad 
\mathcal{R}(x)=(s_x\otimes \id)\mathcal{R}\,,
\end{align*}
where $h$ counts the weights and $s_x$ is an automorphism 
$s_x(X)=x^{k}X$ for elements $X$ of homogeneous degree $k$. 
As an ``auxiliary space'' $V$ we can choose 
a representation of $U_q\widehat{\mathfrak{sl}}_2$, 
or that of its Borel subalgebra. 
The most basic choices are $V=\C^2$, the evaluation two dimensional
module, and $V=M^+$, the so-called prefundamental module of the 
Borel subalgebra. 
The elements $T(x)=T_{\C^2}(x)$ and $Q(x)=T_{M^+}(x)$
satisfy the relation \cite{BLZ,FH,FJMM3} 
\begin{align}
T(x)Q(x)=Q(q^2x)+Q(q^{-2}x)\,. \label{TQ}
\end{align}

Let $W_\lambda$ be the irreducible $U_q\mathfrak{sl}_2$ module with
highest weight $q^\lambda$, $\lambda\in\C$,
and let $W_{\lambda,x}$ be the evaluation module given
by the evaluation map 
$\mathrm{ev}_x:U_q\widehat{\mathfrak{sl}}_2\to U_q\mathfrak{sl}_2$, 
 $x\in \C^\times$.
Consider the tensor product
\begin{align}
W=W_{\lambda_1,x_1}\otimes\cdots\otimes W_{\lambda_N,x_N}\,.
\label{sl2-module}
\end{align}
The $T$ and $Q$ operators,
when appropriately normalized by scalar factors, $T_W(x)=f_W(x)T(x)$, $Q_W(x)=g_W(x)Q(x)$,  
act on each weight subspace of $W$ as polynomials in $x$. 
For generic $x_1,\ldots,x_N$ and $q,p$, their action
is diagonalizable with simple spectrum. 
The TQ relation \eqref{TQ} then allows us to associate with 
each eigenvector $w\in W$ 
a scalar $q$-difference equation with polynomial coefficients
\begin{align}
a_0(x)y(q^4x)-a_1(x)y(q^2x)+a_2(x)y(x)=0\,.
\label{TQdiff}
\end{align}
Here $y(x)$ is the eigenvalue of $Q_W(x)$ on $w$, 
$a_1(x)$ is that of $T_W(q^2x)$, 
and $a_0(x),a_2(x)$ are related to the
highest $\ell$-weight of $W$:
\begin{align*}
pq^{-2 l}\frac{a_2(q^{-2}x)}{a_0(q^{-2}x)}=
\prod_{\nu=1}^N
\frac{q^{\lambda_\nu}x-x_{\nu}} {x-q^{\lambda_\nu}x_{\nu}}\,,
\qquad l=\deg y(x)\,. 
\end{align*}
Let $L$ denote the $q$-difference oper on the left hand side of 
\eqref{TQdiff}.

In the traditional approach, we derive from
\eqref{TQdiff} the Bethe ansatz equations
satisfied by the roots  $t_1,\ldots,t_l$ of $y(x)$,
\begin{align*}
p\prod_{j=1}^l\frac{q^2t_i-t_j}{t_i-q^2t_j} 
=-\prod_{\nu=1}^N\frac{q^{\lambda_\nu}t_i-x_\nu}{t_i-q^{\lambda_\nu}x_\nu}\,,
\quad 1\le i\le l\,.
\end{align*}
Alternatively we can characterize the $q$-oper $L$
directly without making reference to the Bethe roots $\{t_j\}$.

First, $L$ has regular singularities at $x=0,\infty$, as well as 
at the zeroes of $a_0(x)$ and $a_2(x)$ which are given from 
the highest $\ell$-weights:
\begin{align}
a_0(q^{\lambda_\nu-2}x_\nu)=a_2(q^{-\lambda_\nu-2}x_\nu)=0\,,
\qquad 1\le \nu\le N\,.
\label{a0a2}
\end{align}
We will review the basics of linear $q$-difference equations with 
regular singularities in the next section. 

An additional  condition occurs   
if $\lambda_\nu$ is a non-negative integer.
In view of \eqref{a0a2}, we obtain from \eqref{TQdiff} 
a $(\lambda_\nu+1)\times (\lambda_\nu+1)$
linear system of equations for the vector
$^{t}(y(q^{\lambda_\nu}x_\nu),y(q^{\lambda_\nu-2}x_\nu),\ldots,
y(q^{-\lambda_\nu}x_\nu))$, 
and thus the determinant of the coefficient matrix vanishes:
$\Delta_{\lambda_\nu}(q^{\lambda_\nu-2}x_\nu)=0$,
where we set 
\begin{align}
\Delta_r(x) 
=\left|
\begin{matrix}
a_1(x)& a_2(x)&  & & \\
a_0(q^{-2}x)& a_1(q^{-2}x)&a_2(q^{-2}x)&&\\
&&&&\text{{\huge $0$}}\\
&\qquad\ddots&\qquad\ddots&\qquad\ddots& \\
&&&&\\
  &\text{\Huge{$0$}}   &\quad\ddots &a_1(q^{-2r+2}x) &a_2(q^{-2r+2}x)\\
    &      &            &    a_0(q^{-2r}x)&a_1(q^{-2r}x)\\
\end{matrix}
\right|\,.
\label{tridg}
\end{align}
In the next section we will explain 
that this condition 
is the $q$-analog of the monodromy-free condition for Fuchsian differential
equations.   
We say that the points $q^{\pm\lambda_\nu-2}x_\nu$ 
with $\lambda_\nu\in\Z_{\ge0}$
are apparent singularities of the $q$-oper $L$. 
We note in passing that $\Delta_r(x)$ is 
the eigenvalue of a suitably normalized 
transfer matrix corresponding to the auxiliary space $W_{r+1,q^{-r+1}x}$. 

Thus, as far as finite dimensional modules are concerned,
the Bethe ansatz method can be reinterpreted as a
bijective correspondence between the joint
eigenspaces of the Bethe algebra, and
$q$-difference opers with regular singularities all of 
which are apparent except $0,\infty$.
In the case of differential opers, 
an analogous statement
has been proved for Gaudin models associated with simple Lie algebras. 
For type A it is done 
in \cite{MTV2}  and the general case follows from \cite{Ry}, \cite{L}. For the XXX case and difference opers with additive shifts, a similar result is known for generic tensor product of evaluation vector representations of of $\gl_n$ Yangian, \cite{MTV3} and also for the homogeneous $\gl_2$ spin chain, see \cite{MTV1}.

The affinization of this correspondence has been studied intensively under
the name of ODE/IM correspondence starting with the work 
\cite{DT}, \cite{BLZ2}. 
See e.g. \cite{DDT} for a review, 
and \cite{FF,LVY} for a representation theory approach. 
Nevertheless, affine opers are a subject that is not fully 
understood  in the differential case,   
and much less so in the $q$-deformed case. 
In the present paper we make a modest step toward understanding the latter.

It appears that the notion of apparent singularities in the $q$-deformed
case is not discussed much in the literature
\footnote{While preparing  this manuscript, 
we learned about the paper \cite{BJ}
where a notion closely related to Definition \ref{def:apparent} 
is introduced and studied for additive difference equations.},
and one of our aims is to elucidate this point.
We then take up the Bethe ansatz equations for the
quantum toroidal $\mathfrak{gl}_2$ algebra in the evaluation
Verma modules, and reformulate them in terms of $q$-opers.
By extrapolation to the evaluation relaxed Verma modules \cite{JM},
we formulate a conjectural correspondence
between the joint spectra of the Bethe algebra and $q$-hypergeometric
opers with apparent singularities. 
Our study has been motivated partly by the remarkable paper 
\cite{BL} where perturbed (or $\lambda$-) hypergeometric opers 
play a central role.
To find the $q$-deformed counterpart of perturbed opers 
is an important problem left for the future work.

\medskip

This article is organized as follows.
In section \ref{sec:apparent}, we give a review of linear $q$-difference 
equations, and introduce the notion of apparent singularities. 
In section \ref{sec:toroidal}
we set up our notation regarding the quantum toroidal algebras.  
In section \ref{sec:Main} we state our conjecture. 
In section \ref{sec:higher} we examine the possible generalization to 
quantum toroidal $\mathfrak{gl}_n$ algebras.
The text is followed by three appendices. 
In section \ref{app:1} we show that, in the case of second order equations, 
the condition for apparent singularities 
reduces in the limit $q\to 1$ to the standard no-log condition for 
Fuchsian differential equations. 
In section \ref{subsec:connection}, we study the Birkhoff connection 
matrices for $q$-hypergeometric opers with apparent singularities.
Section \ref{sec:unfolded} is a remark concerning $\lambda$-opers.

\section{Apparent singularities of linear 
$q$-difference equations}\label{sec:apparent}

Here we summarize basic facts about linear difference equations
with regular singularities. 
General theory of difference equations 
in the additive and multiplicative ($q$-difference) cases are both  
discussed in the famous work \cite{B}. 
We deal here with the multiplicative case.
We fix $\ssq\in \C^{\times}$ with $|\ssq|<1$, and
let $\sigma_\ssq$ 
denote the $\ssq$-shift operator $(\sigma_\ssq f)(x)=f(\ssq x)$. 
We use the standard symbols
$\theta_\ssq(x)=(x,\ssq/x,\ssq;\ssq)_\infty$,
$(z_1,\ldots,z_r;\ssq)_\infty=\prod_{i=1}^r\prod_{k=0}^{\infty}(1-\ssq^kz_i)$.

\subsection{First order systems}\label{sec:first order}

Consider an $n\times n$ system of linear $q$-difference equations
with rational coefficients
\begin{align}
Y(\ssq x)=A(x)Y(x)\,,\quad
A(x)\in GL_n\bigl(\C(x)\bigr)\,,
\label{YAY}
\end{align}
or $(\sigma_{\ssq}-A)Y=0$ for short.
Using the freedom of changing $Y(x)$ to $f(x)Y(x)$ with a suitable 
scalar function $f(x)$,
we may assume that $A(x)$ is holomorphic at $x=0,\infty$.
We write $A_0=A(0)$, $A_\infty=A(\infty)$.

We make the following assumptions:
\begin{align}
&A_0,A_\infty\in GL_n(\C)\,,
\label{Fuchs}
 \\
&\text{$A_0$, $A_\infty$ are strongly non-resonant}.
\label{nonres}
\end{align}
Here an $n\times n$ matrix $A$ with eigenvalues
$\{\lambda_1,\ldots,\lambda_n\}$ is said to be   
{\it strongly non-resonant} if 
$\lambda_i\not\in q^{\Z}\lambda_j$ for $i\neq j$.
The system \eqref{YAY} is said to 
have {\it regular singularities}\footnote{in \cite{S} it is called Fuchsian}
at $x=0,\infty$ if it satisfies \eqref{Fuchs}.  

For a matrix valued meromorphic function $f(x)$ on $\C^{\times}$,
let $\Pole(f(x))$ denote the set of its poles.
Elements of $\Pole(A(x)^{\pm1})$ are called {\it intermediate singularities} 
of \eqref{YAY}.
We regard them also as regular singularities of \eqref{YAY}.

When $A(x)$ is a constant matrix $C\in GL_n(\C)$, we choose and fix 
a solution $e_C(x)$ of \eqref{YAY} as follows.
We need only the case when $C$ is semi-simple, 
see \cite{S} for the general case. Writing 
$C=G\Lambda G^{-1}$ with $\Lambda=\diag(c_1,\ldots,c_n)$, we set 
\begin{align}
&e_C(x)=G\diag(e_{c_1}(x),\ldots,e_{c_n}(x))G^{-1}\,,
\quad e_{c}(x)=\frac{\theta_\ssq(x)}{\theta_\ssq(c x)}\,.
\label{ec}
\end{align}
Then $e_C(x)$ is meromorphic on $\C^\times$, 
$\mathrm{Pole}\bigl(e_C(x)\bigr)=q^{\Z}\{c_1^{-1},\ldots,c_n^{-1}\}$, and 
$e_C(\ssq x)=Ce_C(x)$.

The system \eqref{YAY} admits 
unique solutions of the following form:
\begin{align*}
&Y_0(x)=\widehat{Y}_0(x)\cdot e_{A_0}(x)\,,
\quad 
\widehat{Y}_0(x)\in GL_n\bigl(\C\{x\}\bigr)\,,
\quad \widehat{Y}_0(0)=I\,,
\\
&Y_\infty(x)=\widehat{Y}_\infty(x)\cdot e_{A_\infty}(x)\,,
\quad 
\widehat{Y}_\infty(x)\in GL_n\bigl(\C\{x^{-1}\}\bigr)\,,
\quad \widehat{Y}_\infty(\infty)=I\,,
\end{align*}
where $\C\{x\}$ denotes the ring of convergent 
power series. 
Moreover the $q$-difference equation \eqref{YAY} allows us to continue 
$\widehat{Y}_0(x)$ (resp. $\widehat{Y}_\infty(x)$)
 meromorphically to
$\CP^1\backslash\{\infty\}$ (resp. $\CP^1\backslash\{0\}$). 
We call $Y_0(x),Y_\infty(x)$ the {\it canonical solutions} of \eqref{YAY}.

A $q$-analog of the notion of monodromy is 
Birkhoff's {\it connection matrix} \cite{B} defined by 
\begin{align*}
M(x)=Y_0(x)^{-1}Y_\infty(x)\,.
\end{align*}
Since $M(\ssq x)=M(x)$,
it is a meromorphic function on $\C^{\times}/\ssq^\Z$, 
that is, an elliptic function.
The matrix
$\hat{M}(x)=\hat{Y}_0(x)^{-1}\hat{Y}_\infty(x)
=e_{A_0}(x) M(x)e_{A_\infty}(x)^{-1}$
is called the central part of the connection matrix.

Set
\begin{align*}
S_\pm=\Pole(A(x)^{\pm1})\ \subset \C^\times\,.  
\end{align*}
By iterating the equation \eqref{YAY} it is easy to see that 
\begin{align*}
&\Pole(\widehat{Y}_0(x)^{\pm1})\subset  \ssq^{\Z_{\le0}}S_\mp\,,
\quad
\Pole(\widehat{Y}_\infty(x)^{\pm1})\subset  \ssq^{\Z_{>0}}S_\pm\,,
\\
&\Pole (\hat{M}(x)^{\pm 1})\subset q^\Z S_{\pm}\,. 
\end{align*}
The latter says that, apart from the
poles of $e_{A_0}(x)$ and $e_{A_\infty}(x)$,  
the singularities of the connection matrix are confined
to intermediate singularities mod $q^{\Z}$. In general the inclusion is strict. 

Let us say that a matrix valued function is {\it regular} at $x=s$ if it is
holomorphic and invertible there.
\begin{dfn}\label{def:apparent}
We say that $s\in S_+\cup S_-$ is an {\rm apparent singularity} 
of \eqref{YAY} if the product
\begin{align}
&A(\ssq^{N}x) A(\ssq^{N-1} x)
\cdots A(\ssq^{-N'}x)
\label{apparent}
\end{align}
is regular at $x=s$ for sufficiently large $N,N'\in\Z_{\ge0}$.
\qed
\end{dfn}
In \eqref{apparent}, 
it is enough to take $N,N'$ so that 
$\ssq^{j} s\not\in S_+\cup S_-$ for $j>N$ or $j< -N'$.

In other words, a singularity $x=s$ is apparent if and only if there exists a fundamental solution $Y(x)$ such that $sq^{\pm N}\not \in Pole(Y(x)^{\pm 1})$ for all sufficiently large $N$.

The following Lemma may be viewed as a justification of this definition. 
\begin{lem}
The  central part of the connection matrix $\hat{M}(x)$  
is regular at an apparent singularity. 
In particular, the connection matrix $M(x)$  is regular at an apparent singularity provided
$S_+\cup S_-$ is disjoint from 
$\mathrm{Pole}\bigl(e_{A_0}(x)\bigr)\cup 
\mathrm{Pole}\bigl(e_{A_\infty}(x)\bigr)$.
\end{lem}
\begin{proof}

By the assumption, 
\begin{align*}
&\hat{Y}_\infty(\ssq^{N+1} x)
=A(\ssq^{N}x)A(\ssq^{N-1}x)\cdots A(\ssq^{-N'}x)\hat{Y}_\infty(\ssq^{-N'}x)
A_\infty^{-N-N'-1}
\end{align*}
is regular at $x=s$ if $N,N'$ are chosen large enough. 
Also $\hat{Y}_0(\ssq^{N+1} x)$ is regular at $x=s$ if $N$ is large enough.
Therefore $\hat{M}(x)=A_0^{-N-1}\hat{M}(\ssq^{N+1} x)A_\infty^{N+1}
=A_0^{-N-1}\hat{Y}_0(\ssq^{N+1} x)^{-1}\hat{Y}_\infty(\ssq^{N+1} x)
A_\infty^{N+1}$ 
is regular 
at $s$.
\end{proof}
\medskip

\noindent{\it Example.}\quad 
Consider a scalar $q$-difference equation 
$Y(\ssq x)=\frac{\ssq^ax-1}{x-1} Y(x)$. 
Then it has singularities at $x=1$ and $x=\ssq^{-a}$. 
These singularities are apparent if $a\in\Z$ and not apparent otherwise. 
We have 
$$
Y_0(x)=\frac{(x;\ssq)_\infty}{(\ssq^a x;\ssq)_\infty}\,,
\quad
Y_\infty(x)=\frac{(\ssq^{1-a}/x;\ssq)_\infty}{(\ssq/x;\ssq)_\infty}
e_a(x)\,,\quad M(x)=1\,.
$$
If $a\in\Z$, then $Y_0(x)=Y_\infty(x)$
becomes the polynomial $\prod_{i=0}^{a-1}(q^ix-1)$ when $a\geq 0$ 
and 
the rational function $\prod_{i=a}^{-1}(q^ix-1)^{-1}$ when $a<0$.
\qed 

\medskip

\subsection{Scalar operators}\label{sec:scalar}

Consider now a scalar $q$-difference operator
with polynomial coefficients
\begin{align}
L=a_0(x)\sigma_{\ssq}^n-a_1(x)\sigma^{n-1}_{\ssq}+\cdots+(-1)^na_n(x)\,.
\label{scalar}
\end{align}
It can be cast into the form of a first order matrix operator 
$\sigma_{\ssq}-A$ 
by setting
\begin{align}
A(x)= 
\begin{pmatrix}
t_1(x) &-t_2(x) & & \cdots &  &(-1)^{n-1}t_n(x)\\
1 &0 &0& \cdots & 0 &0\\
0& 1& 0& \cdots & 0 &0  \\
0& 0& \ddots  &\ddots  &0  & 0 \\
0 &0  &0  &        &0&0\\
0&0 &0 &        &1& 0\\
\end{pmatrix}\,,
\quad t_k(x)=\frac{a_k(x)}{a_0(x)}\,.
\label{scalar-mat}
\end{align}
Note that $\det A(x)=t_n(x)$.
We assume that \eqref{scalar-mat} satisfies
the conditions \eqref{Fuchs}, \eqref{nonres}, 
and define apparent singularities of \eqref{scalar} to be those of 
the system \eqref{YAY} with \eqref{scalar-mat}. 
Hereafter we set 
\begin{align*}
A^{(m)}(x)=A(x)A(\ssq^{-1}x)\cdots A(\ssq^{-m}x)\,,\quad m\ge0\,.
\end{align*}

The following gives a criterion of a singularity to be apparent 
in the case of second order operators.  
\begin{prop}\label{prop:2nd}
Consider a second order $q$-oper
\begin{align}
L=a_0(x)\sigma_{\ssq}^2-a_1(x)\sigma_\ssq+a_2(x)\,.
\label{2nd ord}
\end{align}
Let $s\in\C^{\times}$, $r\in \Z_{>0}$.  
Assume that $a_0(x),a_2(\ssq^{-r}x)$ have a simple zero at $x=s$,
and that there are no other zeroes of the form $\ssq^js$, $j\in\Z$.  
Then the singularities $x=s,\ssq^{-r}s$ 
of \eqref{2nd ord} are apparent  if and only if 
$\Delta_r(s)=0$, where $\Delta_r(x)$ is the $(r+1)\times(r+1)$ tridiagonal determinant
\begin{align}
\Delta_r(x) 
=\left|
\begin{matrix}
a_1(x)& a_2(x)&  & & \\
a_0(\ssq^{-1}x)& a_1(\ssq^{-1}x)&a_2(\ssq^{-1}x)&&\\
&&&&\text{{\huge $0$}}\\
&\qquad\ddots&\qquad\ddots&\qquad\ddots& \\
&&&&\\
  &\text{\Huge{$0$}}   &\quad\ddots &a_1(\ssq^{-r+1}x) &a_2(\ssq^{-r+1}x)\\
    &      &            &    a_0(\ssq^{-r}x)&a_1(\ssq^{-r}x)\\
\end{matrix}
\right|\,.
\label{tridiag}
\end{align}
\end{prop}

\begin{proof}
By the assumption, $A(x)$ is regular at $x=\ssq^js$ if $j\neq 0,-r$. 
Moreover 
\begin{align*}
\det A^{(r)}(x)=
\prod_{j=0}^r t_2(\ssq^{-j} x)=\frac{a_2(\ssq^{-r}x)}{a_0(x)}
\prod_{j=1}^r \frac{a_2(\ssq^{-j+1}x)}{a_0(\ssq^{-j}x)}
\end{align*}
is holomorphic and non-zero at $x=s$. 
Therefore $x=s,\ssq^{-r}s$ are apparent singularities if and only if
$A^{(r)}(x)$ is holomorphic at $x=s$.

By induction it is easy to see that the entries of $A^{(m)}(x)$ are given by
\begin{align*}
&(A^{(m)}(x))_{1,1}
=\frac{\Delta_m(x)}{a_0(x)}
\frac{1}{\prod_{j=1}^m a_0(\ssq^{-j}x)}\,,
\quad
(A^{(m)}(x))_{1,2}
=-\frac{a_2(\ssq^{-m}x)}{a_0(x)}
\frac{\Delta_{m-1}(x)}{\prod_{j=1}^m a_0(\ssq^{-j}x)}\,,
\\
&(A^{(m)}(x))_{2,1}
=
\frac{\Delta_{m-1}(\ssq^{-1}x)}{\prod_{j=1}^m a_0(\ssq^{-j}x)}\,,
\quad
(A^{(m)}(x))_{2,2}
=-\frac{a_2(\ssq^{-m}x)}{a_0(\ssq^{-m}x)}
\frac{\Delta_{m-2}(\ssq^{-1}x)}{\prod_{j=1}^{m-1}a_0(\ssq^{-j}x)}\,,
\end{align*}
where we set $\Delta_{-1}(x)=1$, $\Delta_{-2}(x)=0$.

From the assumption and the formulas above with $m=r$, 
it is clear that 
$(A^{(r)}(x))_{i,j}$ is holomorphic at $x=s $ for $(i,j)\neq(1,1)$, 
and that
$(A^{(r)}(x))_{1,1}$ is holomorphic at $x=s$ if and only if $\Delta_r(s)=0$.
\end{proof}

In appendix \ref{app:1} we will show that, in the limit $\ssq\to1$, 
the condition $\Delta_r(s)=0$ reduces to the 
no-log condition for the limiting differential equation.
This gives another justification of Definition \ref{def:apparent}.
\bigskip

Now we return to the $n$-th order equation \eqref{scalar}--\eqref{scalar-mat}. 

We extend the definition of $t_k(x)$ by $t_k(x)=0$ for $k<0$ or $k>n$. 
For $k,m\in\Z$, define $t^{(m)}_k(x)$ by setting $t^{(0)}_k(x)=t_k(x)$ and
\begin{align}
&t^{(m)}_k(x)
=
\left|
\begin{array}{ccccc|c}
t_1(x) & t_2(x) & \cdots &&t_m(x) & t_{k+m}(x) \\
  1  & t_1(\ssq^{-1}x)&\cdots&&t_{m-1}(\ssq^{-1}x)&t_{k+m-1}(\ssq^{-1}x)\\
    &\ddots      & \ddots      &&           \vdots     &\vdots\\
&&&&&\\
    &\text{\rm{\huge{0}}}  &    &1   &t_1(\ssq^{-m+1}x)&t_{k+1}(\ssq^{-m+1}x)\\
\hline
  0  &     \dots      &     &  0 & 1            &t_k(\ssq^{-m}x)\\
\end{array}
\right|\,,
\quad \text{if $m> 0$}\,,
\label{tNk}
\\
&t^{(m)}_k(x)=(-1)^{k-1}\delta_{k+m,0}
\quad \text{if $m< 0$}\,.
\nn
\end{align}

Note the recurrence relation
\begin{align*}
t^{(m)}_k(x)=t_1^{(m-1)}(x)t_k(\ssq^{-m}x)-t^{(m-1)}_{k+1}(x)\,.
\end{align*}
In particular, since $t^{(m)}_k(x)$=0 for $k>n$ we have 
\begin{align}
t^{(m)}_n(x)=t^{(m-1)}_1(x)t_n(\ssq^{-m}x)\,. 
\label{tmn}
\end{align}

\begin{lem}\label{lem:AAA}
For $m\ge0$ we have
\begin{align}
\bigl(A^{(m)}(x)\bigr)_{j,k}=(-1)^{k-1}t^{(-j+m+1)}_{k}(\ssq^{-j+1}x)\,.
\label{AAAt}
\end{align}
\end{lem}
\begin{proof}
This can be shown by induction on $m$.
\end{proof}

\begin{prop}\label{prop:apparent-nth}
Consider an $n$-th order $q$-oper \eqref{scalar}, 
and let $s\in\C^\times, r,t\in\Z$, $0\le t<r$.
Assume that 
$a_0(x)$ has simple zeroes at $x=\ssq^{-j}s$, $0\le j\le t$, 
$a_n(x)$ has simple zeroes at $x=\ssq^{-j}s$, $r-t\le j\le r$, 
and that there are no other zeroes of the form $\ssq^js$, $j\in\Z$.  

Then the singularities $x=\ssq^{-j}s$ with $0\le j\le t$ or $r-t\le j\le r$  of  \eqref{scalar} are
apparent if and only if 
\begin{align}
&
t^{(-j+r+1)}_k(\ssq^{-j+1}x)\,,\quad 1\le j\le t+1\,,\ 1\le k\le n ,
\quad 
\text{
are holomorphic at $s$. }
\label{app-sc1}
\end{align}
\end{prop}
\begin{proof}
By the assumption, $A(x)$ is regular at $x=\ssq^js$ if $j>0$ or $j<-r$.
We have 
\begin{align*}
\det A^{(r)}(x)=\prod_{j=0}^r\frac{a_n(\ssq^{-j}x)}{a_0(\ssq^{-j}x)} 
=\prod_{j=0}^t\frac{a_n(\ssq^{-j-r+t}x)}{a_0(\ssq^{-j}x)}
\prod_{j=t+1}^r\frac{a_n(\ssq^{j-r}x)}{a_0(\ssq^{-j}x)}
\end{align*}
is holomorphic and nonzero at $x=s$. Therefore the points in question
are apparent if and only if $A^{(r)}(x)$ is holomorphic at $x=s$.
According to \eqref{AAAt}, its $(j,k)$-th entry is 
$\pm t^{(-j+r+1)}_k(\ssq^{-j+1}x)$. 
The denominator of the latter $\prod_{i=j-1}^ra_0(\ssq^{-i}x)$ 
does not vanish at $x=s$ provided $j>t+1$. 
\end{proof}

The following special case will be used later.
\begin{prop}\label{prop:apparent-special}
Consider an $n$-th order $q$-oper \eqref{scalar} 
and let $s\in\C^\times$. 
Assume that 
$a_0(x)$ has simple zeroes at $x=\ssq^{-j}s$, $0\le j\le n-2$, 
$a_n(x)$ has simple zeroes at $x=\ssq^{-j}s$, $1\le j\le n-1$, 
and that there are no other zeroes of the form $\ssq^js$, $j\in\Z$.
Then the singularities $x=\ssq^{-j}s$, $0\le j\le n-1$, are apparent if the following conditions are satisfied.
\begin{enumerate}
\item For $1\le k\le n-1$, 
$t_k(x)$ is holomorphic at $x=\ssq^{-j}s$ with $j\neq n-k,n-k-1$. 
\item For $2\le k\le n-1$, $t_k(\ssq^{-n+1}s)=0$.
\item
For $1\le k\le n-1$ the determinant
$$
\left|\begin{matrix}
      t_k(\ssq^{-n+k+1}x)& t_{k+1}(\ssq^{-n+k+1}x)\\
      1 & t_1(\ssq^{-n+1}x)\\
     \end{matrix}
\right|
$$
is holomorphic at $x=s$.
\end{enumerate}
\end{prop}
\begin{proof}
The assumption implies that $t_n(x)$ is holomorphic at $x=\ssq^{-j}s$ with 
$j\neq 0$, and $t_n(\ssq^{-n+1}s)=0$.

By Proposition \ref{prop:apparent-nth}, it suffices to show that 
$t_k^{(-j+n)}(\ssq^{-j+1}x)$ is holomorphic at $x=s$ for $1\le j\le n-1$,
$1\le k\le n$. We use induction on $j=n-1,\ldots,1$ based on 
the recursion relation
\begin{align*}
t_k^{(-j+n)}(\ssq^{-j+1}x)&=\sum_{l=1}^{n-j-1}(-1)^{l-1}
t_l(\ssq^{-j+1}x)t_k^{(n-j-l)}(\ssq^{-j-l+1}x)\\
&+(-1)^{n-j-1}
\left|\begin{matrix}
      t_{n-j}(\ssq^{-j+1}x)& t_{n-j+k}(\ssq^{-j+1}x)\\
      1 & t_k(\ssq^{-n+1}x)\\
     \end{matrix}
\right|\,.
\end{align*}
Note that the second term in the right hand side 
is holomorphic by (iii) if $k=1$, 
and by (ii) if $2\le k\le n$.

The assertion is true for $j=n-1$ where the first term 
is absent. 
Suppose the assertion is true for $j>j_0$. 
Since $t_l(\ssq^{-j_0+1}x)$ does not have a pole by (i),
the first term is holomprphic by the induction hypothesis.

This completes the proof.
\end{proof}

\section{Quantum algebras}\label{sec:toroidal}

In this section we explain our notation and terminology about
quantum affine and quantum toroidal algebras which will be used in the next section. 
We do not try to be self-contained here.
For the defining relations of the algebras and other details, 
the reader is referred to \cite{JM}.

Throughout this paper we fix $q,d\in\C^{\times}$ and set
$q_1=q^{-1}d$, $q_2=q^2$, $q_3=q^{-1}d^{-1}$ so that $q_1q_2q_3=1$. 
We assume that $q_1^aq_2^b=1$ with $a,b\in\Z$ implies $a=b=0$. 
We assume also $|q|<1$. 
Later we will use $\ssq=q_2$ as the $\ssq$-shift parameter.

\subsection{Quantum affine algebra $U_q\widehat{\mathfrak{gl}}_2$}\label{sec:quantum affine}

The quantum algebra $U_q\mathfrak{sl}_2$ 
has generators $e_1,f_1,K^{\pm1}$ with standard relations.
The Casimir element $\mathcal{C}=qK+q^{-1}K^{-1}+(q-q^{-1})^2f_1e_1$ is central.
The algebra  $U_q\mathfrak{gl}_2$ is defined by adjoining
an invertible split central element $\mathfrak{t}$ to  $U_q\mathfrak{sl}_2$. 

The quantum affine algebra $U_q\widehat{\mathfrak{sl}}_2$ 
is generated by the coefficients of the generating series 
\begin{align*}
x^\pm(z)=\sum_{k\in\Z}x^\pm_k z^{-k}\,, 
\quad
\phi^\pm(z)=K^{\pm1}\exp\bigl(\pm(q-q^{-1})\sum_{\pm r>0}h_r z^{-r}\bigr)\,,
\end{align*}
together with an invertible central element $C$. 
The defining relations are also standard and are given in Section 2.2 of \cite{JM}. 
We define $U_q\widehat{\mathfrak{gl}}_2$ by adjoining to  $U_q\widehat{\mathfrak{sl}}_2$
a split central element $\mathfrak{t}$, 
and an extra Heisenberg algebra with generators $Z_r$, $r\in\Z\backslash\{0\}$,
which commute with $U_q\widehat{\mathfrak{sl}}_2$. 
Algebra $U_q\mathfrak{gl}_2$ is embedded into $U_q\widehat{\mathfrak{gl}}_2$ by
$e_1\mapsto x^+_0$, $f_1\mapsto x^-_0$, $K\mapsto K$, 
$\mathfrak{t}\mapsto \mathfrak{t}$. 

Algebra $U_q\widehat{\mathfrak{gl}}_2$ has a $\Z^2$-grading with the assignment
\begin{align*} 
\deg x^\pm_k=(\pm1,k), \qquad \deg h_r=\deg Z_r=(0,r), 
\qquad \deg K=\deg \mathfrak{t}=\deg C=(0,0).
\end{align*}

Let $V$ be a $U_q\widehat{\mathfrak{gl}}_2$-module. 
We say that $V$ has level $\kappa\in\C^{\times}$ if $C$ acts on $V$ as a scalar $\kappa$.
We say $V$ is a highest weight module with highest weight $(a,b)\in (\C^\times)^2$ if it has
a cyclic vector $v\in V$ satisfying $x^+_{k}v=x^{-}_{r}v=h_rv=Z_rv=0$ for $k\ge0$, $r>0$,
and $Kv=b v$, $Cv=abv$, 
$\mathfrak{t}v=c v$ for some $c\in\C^{\times}$.

\subsection{Quantum toroidal $\mathfrak{gl}_2$ algebras}\label{sec:toroidal gl2}

We will use two presentations of the quantum toroidal $\mathfrak{gl}_2$ algebra, 
the ``perpendicular realization'' $\E^\perp$, and the ``parallel realization'' $\E$. 
These two presentations are related by the Miki automorphism. 
The realization $\E^\perp$ is used to construct the Bethe algebra, 
whereas the realization $\E$ is used to describe highest weight representations.
Another point about  $\E^\perp$ is that 
there exists a surjective homomorphism $ev_u$
(the evaluation map) from $\E^\perp$ to 
the quantum affine algebra $U_q\widehat{\mathfrak{gl}}_2$ at a special level.

The quantum toroidal $\mathfrak{gl}_2$ algebra in the perpendicular realization 
$\E^\perp$ 
 is generated by the coefficients of the generating series
\begin{align*}
&E^\perp_i(z)=\sum_{k\in \Z}E^\perp_{i,k}z^{-k}\,,
\quad 
F^\perp_i(z)=\sum_{k\in \Z}F^\perp_{i,k}z^{-k}\,,
\\
&K_i^{\pm,\perp}(z)=\bigl(K_i^{\perp}\bigr)^{\pm1}
\exp\bigl(\pm(q-q^{-1})
\sum_{r>0}H^\perp_{i,\pm r}z^{\mp r}\bigr)\,,
\end{align*}
where $i\in\Z/2\Z$, $K_0^\perp=K^{-1}$, $K_1^\perp=K$, 
an invertible central element $C$,
and an invertible split central element $\mathfrak{t}$.
The defining relations depend on parameters $q_1,q_2,q_3$ 
and are given in section 3.1 of \cite{JM}.

Algebra $\E^\perp$ is $\Z^3$-graded with the assignment
\begin{align*}
&\deg E^\perp_{i,k}=(1_i,k)\,,
\quad
\deg F^\perp_{i,k}=(-1_i,k)\,,
\quad \deg H^\perp_{i,r}=(0,0,r), \quad  \deg K=\deg \mathfrak{t}=\deg C=(0,0,0)\,,
\end{align*}
where $1_0=(1,0)$ and $1_1=(0,1)$. 

We have an embedding $v^\perp:\ U_q\widehat{\mathfrak{sl}}_2\to \E^\perp$ such that 
\begin{align*}
&x^+(z)\mapsto E_1^\perp(d^{-1}z)\,,\quad  
x^-(z)\mapsto F_1^\perp(d^{-1}z)\,,\quad  
\phi^\pm(z)\mapsto K^{\pm,\perp}_1(d^{-1}z)\,,
\quad C\mapsto C\, ,\quad \mathfrak{t}\to \mathfrak{t}.
\end{align*}
For $\kappa\in\C^{\times}$, 
denote by $\E_\kappa^\perp$ (resp. $U_{q,\kappa}\widehat{\gl}_2$)
the quotient of $\E^\perp$ (resp. $U_{q}\widehat{\gl}_2$) 
by the relation $C=\kappa$. When $\kappa=q_3$ there exists
the following evaluation map due to Miki, see Theorem 4.1 in \cite{FJM2}:
\begin{prop}
Let $u\in \C^{\times}$. There is a surjective homomorphism of algebras 
\begin{align}
{ev}_u:\E^\perp_{q_3}\longrightarrow  \widetilde{U}_{q,q_3}\widehat{\gl}_2
\label{eval}
\end{align}
such that $ev_u\circ v^\perp=\id$, 
and $\deg \bigl(ev_u(E^\perp_{0,k})\bigr)=(-1,k)$,
$\deg \bigl(ev_u(F^\perp_{0,k})\bigr)=(1,k)$. 
Here $\widetilde{U}_{q,\kappa}\widehat{\gl}_2$ denotes a completion
of $U_{q,\kappa}\widehat{\gl}_2$.
\end{prop}

The quantum toroidal $\mathfrak{gl}_2$ algebra in the parallel realization $\E$ 
is generated by the coefficients of  the generating series
\begin{align*}
&E_i(z)=\sum_{k\in\Z}E_{i,k}z^{-k}\,,\quad 
F_i(z)=\sum_{k\in\Z}F_{i,k}z^{-k}\,,\quad \\
&K^{\pm}_i(z)=K_i^{\pm1}\exp\bigl(
\pm(q-q^{-1})\sum_{\pm r>0}H_{i,r}z^{-r}\bigr)\,,
\end{align*}
with $i\in\Z/2\Z$, 
together with an invertible split central element $\mathfrak{t}$
(notice the absence of $C$). 
The defining relations are in section 3.3 of \cite{JM}. We have 
a homomorphism $h:U_q\widehat{\mathfrak{sl}}_2\longrightarrow \E$ such that
\begin{align*}
 x^+(z)\mapsto E_1(d^{-1}z)\,,\quad
 x^-(z)\mapsto F_1(d^{-1}z)\,,\quad
\phi^{\pm}(z)\mapsto K_1^{\pm}(d^{-1}z)\,\quad
C\mapsto 1\,.
\end{align*}
The image $U_q^h\widehat{\mathfrak{sl}}_2= 
h\bigl(U_q\widehat{\mathfrak{sl}}_2\bigr)$ 
is isomorphic to 
$U_{q,1}\widehat{\mathfrak{sl}}_2$ with $C=1$.

Algebras $\E^\perp$ and $\E$ are in fact isomorphic, see Proposition 3.2
in \cite{JM} for the precise statement.

Henceforth we will consider only such representations on which 
$\mathfrak{t}$ ats as $1$ without further mention. 
\bigskip

Let $\Psi(z)=(\Psi_0(z),\Psi_1(z))$ be a pair of rational functions
which are regular at $z=0,\infty$ and satisfy  $\Psi_i(0)\Psi_i(\infty)=1$. 
An $\E$-module $W$ is called a highest weight module 
with highest $\ell$-weight $\Psi(z)$
if it has a cyclic vector $w\in W$ satisfying 
 \begin{align*}
E_i(z)w=0, \qquad
K^{\pm}_i(z)w=\Psi_i(z)w\,,\qquad i=0,1\,.
\end{align*}
Let $\hat{L}_{a,b}$ be a highest weight 
$U_q\widehat{\mathfrak{gl}}_2$-module with highest weight
$(a,b)$, and let $q_3=ab$ be its level. 
Then the $\E^\perp$-module obtained via the evaluation map 
has the highest $\ell$-weight
\begin{align*}
\Psi(z)=\big(\frac{az-u}{z-au},\  \frac{bz-u}{z-bu}\big).  
\end{align*}
The parameter $u$ depends on the choice of the evaluation map and 
can be made arbitrary by twisting by the automorphism
$X(z)\to X(z/a)$ for $X=E_i,F_i,K^\pm_i$ with some $a\in\C^{\times}$.

\subsection{Bethe algebra}\label{sec:Bethe}
Algebra $\E^\perp$ has a topological Hopf algebra structure 
given by the Drinfeld coproduct. 
Moreover it is a (quotient of a) quantum double,
and hence is endowed with a universal $\cR$ matrix, see \cite{N}. 
Fix $(\pz{0},\pz{1})\in(\C^{\times})^2$ and set $p=\pz{0}\pz{1}$. 
We assume $|p|<1$. The transfer matrices
\begin{align*}
T_V(x)=\Tr_{V}\Bigl((p^{-D} \pz{1}^{-\bar{\Lambda}_1}\otimes\id)
\cR(x)\Bigr)\,
\end{align*}
mutually commute for various choices of highest weight representations
$V$ of $\E^\perp$.
Here for a vector $v\in V$ of degree $(l,d)$ we set 
$p^{-D}p_1^{-\bar\Lambda_1}v=p^{-d}p_1^{-l}v$. 
We call the commutative subalgebra generated by $T_V$'s 
the Bethe algebra of $\E^\perp$, and denote it by
$\cA(\pz{0},\pz{1})$.

Let
$W=\bigoplus_{l,d\in\Z}W_{l,d}$ 
be a graded representation of $\E^\perp$
such that $W_{l,d}=0$ for $d>0$ and $\dim W_{l,d}<\infty$. 
The Bethe algebra acts on each graded component $W_{l,d}$. 
Moreover the restriction of $\cA(\pz{0},\pz{1})$
to the top degree subspaces $W_{l,0}$ is independent of $\pz{0}$, 
and it coincides with the action of the standard Bethe algebra of 
$U_q\widehat{\gl}_2$, see \cite{FT}.  

\section{Bethe algebra and $q$-hypergeometric opers}\label{sec:Main}

A $q$-hypergeometric oper is a second order $q$-oper with linear coefficients
\begin{align*}
\mathcal{L}=(a' x+a'')\sigma_\ssq^2-(b'x+b'')\sigma_\ssq+(c'x+c'')\,,
\end{align*}
where $a',a'',b',b'',c',c''\in \C^\times$.
By dividing by a scalar, rescaling the variable $x\to k x$
and conjugation $\mathcal{L}\to x^{-\lambda}\mathcal{L}x^\lambda$,  
it can be brought to the standard form
\begin{align*}
\mathcal{L}_{HG}=
x(a \sigma_\ssq-1)(b\sigma_{\ssq}-1)
-(c\ssq^{-1}\sigma_\ssq-1)(\sigma_{\ssq}-1)\,,
\end{align*}
which has Heine's series as a solution,
\begin{align*}
{}_2\phi_1(a,b;c;\ssq,x)=\sum_{n=0}^\infty\frac{(a;\ssq)_n(b;\ssq)_n}
{(c;\ssq)_n(\ssq;\ssq)_n} x^n\,,
\qquad (a;\ssq)_n=\prod_{j=0}^{n-1}(1-a\ssq^j)\,.
\end{align*}

In this section we consider a class of evaluation 
$\E^\perp$-modules and discuss their relation to $q$-hypergeometric opers.
To be specific, we consider Verma modules and relaxed Verma modules 
studied in \cite{JM}.
They are irreducible, tame $\E^\perp$-modules equipped with 
combinatorial bases and explicit action of the generators on them.  
These properties will not be used in the following.

\subsection{Evaluation Verma modules}\label{sec:ev_Verma}

Fix $\mu\in\C\backslash\Z$. 
Consider the Verma module
over $U_q\widehat{\mathfrak{gl}}_2$ 
with highest weight $(q^\mu,q_3q^{-\mu})$.
Since the level is $q_3$, it can be viewed as 
a highest weight $\E^\perp$-module via the evaluation map $ev_u$.
We call this $\E^\perp$-module 
the evaluation Verma module and denote it by $G_1^\mu$.
With an appropriate choice of the evaluation parameter $u$, 
$G_1^\mu$ has  highest $\ell$-weight 
\begin{align*}
\Psi^\mu(z)=(\Psi^\mu_0(z),\Psi^\mu_1(z))
=\Bigl(
q^{\mu}\frac{z-q_3^{-1}q_2^{-\mu}}{z-q_3^{-1}},q_3q^{-\mu}\frac{z-q_3^{-2}}{z-q_2^{-\mu}}
\Bigr)\,.
\end{align*}
The module $G_1^\mu$ has a $\Z^2$-grading inherited from the grading of $U_q\widehat{\mathfrak{gl}}_2$,
$G_1^\mu=\bigoplus_{l,d\in\Z}(G_1^\mu)_{l,d}$ such that 
$(G_1^\mu)_{l,d}=0$ for $d>0$, or $d=0,l>0$, with the character
\begin{align*}
\sum_{l,d\in\Z}\dim (G_1^\mu)_{l,d}\, w^l 
t^{-d}
=\frac{1}{\prod_{i=1}^\infty (1-t^{i})^2 
(1-w t^{i})(1-w^{-1} t^{i-1})
}\,.
\end{align*}
Figure  \ref{ev Verma} below shows the dimensions $\dim (G_1^\mu)_{l,d}$ 
in boldface letters.
\vskip 1cm

\begin{figure}[ht]
\centering
\begin{tikzpicture}
\draw[dashed] (-8,0.5) -- (5.5,0.5);
\node at (6,0.5) {$\to l$};
\draw[dashed] (-7.3,1) -- (-7.3,-3.5);
\node at (-7.1,1.5) {$\uparrow d$};
\coordinate (a0) at (-6,0);
\coordinate (a1) at (-4.5,0);
\coordinate (a2) at (-3,0);
\coordinate (a3) at (-1.5,0);
\coordinate (a4) at (0,0);
\coordinate (a5) at (1.5,0);
\coordinate (a6) at (3,0);
\coordinate (a7) at (4.5,0);

\node[above] at ($(a0)+(0,0.6)$) {\tiny $-4$}; 
\node[above] at ($(a1)+(0,0.6)$){\tiny $-3$}; 
\node[above] at ($(a2)+(0,0.6)$) {\tiny $-2$};
\node[above] at ($(a3)+(0,0.6)$) {\tiny $-1$};
\node[above] at  ($(a4)+(0,0.6)$) {\tiny $0$};
\node[above] at ($(a5)+(0,0.6)$) {\tiny $1$};
\node[above] at ($(a6)+(0,0.6)$) {\tiny $2$};
\node[above] at  ($(a7)+(0,0.6)$) {\tiny $3$};

\node at (-7.7,0) {\tiny $0$};
\draw (-7.4,0) -- (-7.2,0);
\node at (-7.7,-0.7) {\tiny $-1$};
\draw (-7.4,-0.7) -- (-7.2,-0.7);
\node at (-7.7,-1.4) {\tiny $-2$};
\draw (-7.4,-1.4) -- (-7.2,-1.4);
\node at (-7.7,-2.1) {\tiny $-3$};
\draw (-7.4,-2.1) -- (-7.2,-2.1);
\node at (-7.7,-2.8) {\tiny $-4$};
\draw (-7.4,-2.8) -- (-7.2,-2.8);

\draw ($(a0)+(0,0.4)$) -- ($(a0)+(0,0.6)$);
\draw ($(a1)+(0,0.4)$) -- ($(a1)+(0,0.6)$);
\draw ($(a2)+(0,0.4)$) -- ($(a2)+(0,0.6)$);
\draw ($(a3)+(0,0.4)$) -- ($(a3)+(0,0.6)$);
\draw ($(a4)+(0,0.4)$) -- ($(a4)+(0,0.6)$);
\draw ($(a5)+(0,0.4)$) -- ($(a5)+(0,0.6)$);
\draw ($(a6)+(0,0.4)$) -- ($(a6)+(0,0.6)$);
\draw ($(a7)+(0,0.4)$) -- ($(a7)+(0,0.6)$);

\node at (a0) {$\dots$};
\node at ($(a0)-(0,2.7)$) {$\vdots$};

\node at (a1) {\tiny${\bf 1}$};
\node at ($(a1)-(0,0.7)$) {\tiny${\bf 4}$};
\node at ($(a1)-(0,1.4)$) {\tiny${\bf 14}$};
\node at ($(a1)-(0,2.1)$) {\tiny${\bf 40}$};
\node at ($(a1)-(0,2.7)$) {$\vdots$};

\node at (a2) {\tiny${\bf 1}$};
\node at ($(a2)-(0,0.7)$) {\tiny${\bf 4}$};
\node at ($(a2)-(0,1.4)$) {\tiny${\bf 14}$};
\node at ($(a2)-(0,2.1)$) {\tiny${\bf 39}$};
\node at ($(a2)-(0,2.7)$) {$\vdots$};

\node at (a3) {\tiny${\bf 1}$};
\node at ($(a3)-(0,0.7)$) {\tiny${\bf 4}$};
\node at ($(a3)-(0,1.4)$) {\tiny${\bf 13}$};
\node at ($(a3)-(0,2.1)$) {\tiny${\bf 36}$};
\node at ($(a3)-(0,2.7)$) {$\vdots$};

\node at (a4) {\tiny${\bf 1}$};
\node at ($(a4)-(0,0.7)$) {\tiny${\bf 3}$};
\node at ($(a4)-(0,1.4)$) {\tiny${\bf 10}$};
\node at ($(a4)-(0,2.1)$) {\tiny${\bf 27}$};
\node at ($(a4)-(0,2.7)$) {$\vdots$};

\node at ($(a5)-(0,0.7)$) {\tiny${\bf 1}$};
\node at ($(a5)-(0,1.4)$) {\tiny${\bf 4}$};
\node at ($(a5)-(0,2.1)$) {\tiny${\bf 13}$};
\node at ($(a5)-(0,2.7)$) {$\vdots$};

\node at ($(a6)-(0,1.4)$) {\tiny${\bf 1}$};
\node at ($(a6)-(0,2.1)$) {\tiny${\bf 4}$};
\node at ($(a6)-(0,2.7)$) {$\vdots$};

\node at ($(a7)-(0,2.1)$) {\tiny${\bf 1}$};
\node at ($(a7)-(0,2.7)$) {$\vdots$};

\node at (-6,-1) {$\dots$};
\node at (-6,-1.5) {$\dots$}; 
\node at (-6,-2) {$\dots$};

\end{tikzpicture}
\caption{Grading in the evaluation Verma module $G_1^\mu$}
\label{ev Verma}
\end{figure}

\vskip1cm

The Bethe algebra $\mathcal{A}(p_0,p_1)$ act on each graded component $(G_1^\mu)_{l,d}$.
We expect the following to be true.

\begin{conj}\label{conj:2}
Assume that all parameters $q,q_3,\mu,\pz{0},\pz{1}$ are generic.
Then for each joint eigenvector of 
$\cA(\pz{0},\pz{1})$ on $G_1^\mu$
of degree $(l,d)=(l_0-l_1,-l_0)$, 
there exist polynomials
$y_0(x)=\prod_{i=1}^{l_0}(x-s_i)$, 
$y_1(x)=\prod_{i=1}^{l_1}(x-t_i)$
with non-zero distinct roots, 
satisfying the Bethe ansatz equations
\begin{align}
& \pz{0}q_2^{l_1-l_0}
\frac{y_0(q_2s_i)}{y_0(q_2^{-1}s_i)}
\frac{y_1(q_1s_i)y_1(q_3s_i)}{y_1(q_1^{-1}s_i)y_1(q_3^{-1}s_i)}
 =-\Psi^\mu_0(s_i)\,,\quad 1\le i\le l_0\,,
\label{BE0}\\
& \pz{1}q_2^{l_0-l_1}
\frac{y_1(q_2t_i)}{y_1(q_2^{-1}t_i)}
\frac{y_0(q_1t_i)y_0(q_3t_i)}{y_0(q_1^{-1}t_i)y_0(q_3^{-1}t_i)}
=-\Psi^\mu_1(t_i)\,, \quad 1\le i\le l_1
\,.
\label{BE1}
\end{align}
Furthermore, none of the factors appearing in \eqref{BE0}, \eqref{BE1} vanish.
\end{conj}

For tensor products of Fock modules 
this type of result has been stated in \cite{AO}, 
though the full details are not clear to us.
We note also that, 
in the case of the quantum toroidal $\mathfrak{gl}_1$ algebras, 
Bethe ansatz equations have been proved in \cite{FJMM1}, \cite{FJMM2} in two different ways. We expect the same proofs work for quantum toroidal $\mathfrak{gl}_2$, see \cite{FJM}, Conjecture 5.4.
In this paper we will make no attempt at proving 
conjecture \ref{conj:2}; instead we
focus on reinterpreting the solutions of the Bethe equations 
in terms of $q$-difference opers. 

For the moment we regard the $s_i$'s as given parameters
and consider the Bethe equations \eqref{BE1}. 
Rewriting them as
\begin{align*}
p_1\prod_{j=1}^{l_1}\frac{q^2t_i-t_j} {t_i-q^2t_j}
=-\frac{q_3q^{-\mu}t_i-q_3^{-1}q^{-\mu}}{t_i-q^{-2\mu}}
\prod_{j=1}^{l_0}\frac{q t_i-d s_j}{t_i-q d s_j}
\frac{q t_i-d^{-1}s_j}{t_i-q d^{-1} s_j}\,,
\end{align*}
we see that they take the form of the Bethe equations for
 $U_q\mathfrak{sl}_2$ associated with the module 
\begin{align*}
W_{\lambda_1,x_1}\otimes\Bigl(\bigotimes_{j=1}^{l_0}
\bigl(\C^2_{d s_j}\otimes \C^2_{d^{-1}s_j}\bigr)\Bigr) \,,
\qquad 
q^{\lambda_1}=q_3q^{-\mu}\,,\ x_1=q_3^{-1}q^{-\mu}\,,
\end{align*}
in the notation of \eqref{sl2-module} (we recall that $q_1=q^{-1}d$, $q_3=q^{-1}d^{-1}$).
The corresponding $U_q\mathfrak{sl}_2$-oper is given by 
\begin{align}
L&=a_0(x)\sigma_{q_2}^2-a_1(x)\sigma_{q_2}+a_2(x)\,,
\label{oper-Verma}\\
&=a_0(x)\bigl(\sigma_{q_2}-b_2(x)\bigr)\bigl(\sigma_{q_2}-b_1(x)\bigr)\,.\nn
\end{align}
Here 
\begin{align}
&a_0(q_2^{-1}x)=p_1q_2^{-l_0-l_1}(x-q^{\lambda_1}x_1)\prod_{j=1}^{l_0}
(x-q_1^{-1}s_j)(x-q_3^{-1}s_j)\,,
\label{a0a1a2}\\ 
&a_2(q_2^{-1}x)=q_2^{-l_0}(q^{\lambda_1}x-x_1)\prod_{j=1}^{l_0}
(x-q_1s_j)(x-q_3s_j)\,,
\nn\\
&a_1(q_2^{-1}x)=
a_{1,0}x^{2l_0+1}+\sum_{j=1}^{2l_0}a_{1,j}x^{2l_0+1-j}+a_{1,2l_0+1}\,,
\nn
\end{align}
with some coefficients $a_{1,j}$, and
\begin{align*}
&b_1(q_2^{-1}x)=\frac{y_1(x)}{y_1(q_2^{-1}x)}\,,
\quad
b_2(q_2^{-1}x)=p_1^{-1}q_2^{l_1-2l_0}\Psi^{\mu}_1(x)\frac{y_1(q_2^{-1}x)}{y_1(x)}
\frac{y_0(q_1^{-1}x)y_0(q_3^{-1}x)} {y_0(q_1x)y_0(q_3x)}\,. 
\end{align*}
The fact that the equation $Ly=0$ has a polynomial solution $y_1(x)$ of degree $l_1$ 
yields the following expression for the first and the last coefficients of $a_1(x)$:
\begin{align}
&a_{1,0}=p_1q_2^{-l_0}+q_2^{-l_0-l_1}q^{\lambda_1}\,,\quad
a_{1,2l_0+1}=-(p_1q_2^{-l_1}q^{\lambda_1}+q_2^{-2l_0})
x_1 \prod_{j=1}^{l_0} s_j^2\,.\label{a1}
\end{align}
In addition, as was explained in introduction, there are 
apparent singularities corresponding to $\C^2_{d^{\pm1} s_j}$, namely
$x=q_1s_i,q_3s_i$ and $x=q^{-2}q_1s_i,x=q^{-2}q_3s_i$: 
\begin{align}\label{Del13}
&\Delta_1(q_1s_i) =\Delta_1(q_3s_i) =0\,,
\qquad 1\le i\le l_0\,,
\\
&\Delta_1(x)=
\left|
\begin{matrix}
 a_1(x) & a_2(x)\\
a_0(q_2^{-1}x)& a_1(q_2^{-1}x)
\end{matrix} 
\right|\,.\nn
\end{align}
The remaining coefficients $\{a_{1,j}\}_{1\le j\le 2l_0}$ are to be determined from these
$2l_0$ equations (regarding $s_i$'s as given).

Now we turn to the zeroth Bethe equations \eqref{BE0}. Setting
$x=q_1s_i, q_2^{-1}q_1s_i$ in $Ly_1=0$ 
we obtain
\begin{align*}
\frac{y_1(q_1s_i)} {y_1(q_3^{-1}s_i)}=
\frac{a_1(q_1s_i)}{a_2(q_1s_i)}
=\frac{a_0(q_2^{-1}q_1s_i)}{a_1(q_2^{-1}q_1s_i)}
\end{align*}
and a similar relation with $q_1$ and $q_3$ interchanged. This allows us
to eliminate $y_1(x)$ from \eqref{BE0} and rewrite it as
\begin{align}
p_0\prod_{j=1}^{l_0}\frac{q_2s_i-s_j} {s_i-q_2s_j} 
=-q_2^{-l_1}
\frac{q^{\mu}s_i-q_3^{-1}q^{-\mu}}{s_i-q_3^{-1}}
\frac{a_2(q_1s_i)}{a_1(q_1s_i)}\frac{a_2(q_3s_i)}{a_1(q_3s_i)}\,,
\quad 1\le i\le l_0\,.
\label{p0ss}
\end{align}

Thus we have rewritten the Bethe equations \eqref{BE0}, \eqref{BE1} in terms of
$q$-opers \eqref{oper-Verma} satisfying the conditions \eqref{a0a1a2}--\eqref{p0ss}. 
We expect that generically all 
$q$-opers \eqref{oper-Verma} satisfying the conditions \eqref{a0a1a2}--\eqref{p0ss}, appear from solutions of  \eqref{BE0}, \eqref{BE1}.

We remark that \eqref{p0ss} can be written equivalently in either of 
the following form.
\begin{align*}
&p_0\prod_{j=1}^{l_0}\frac{q_2s_i-s_j} {s_i-q_2s_j} 
=-q_2^{-l_1}
\frac{q^{\mu}s_i-q_3^{-1}q^{-\mu}}{s_i-q_3^{-1}}
\frac{a_1(q_2^{-1}q_1s_i)}{a_0(q_2^{-1}q_1s_i)} \frac{a_1(q_2^{-1}q_3s_i)}{a_0(q_2^{-1}q_3s_i)}
\,,\\
&p_0p_1\prod_{j=1}^{l_0}\frac{q_3^2s_i-s_j}{s_i-q_3^2s_j}
=-q_1^{l_0}q_3^{-l_0-1}
\frac{a_1(q_2^{-1}q_3s_i)}{a_1(q_1s_i)}\,,
\\
&p_0p_1\prod_{j=1}^{l_0}\frac{q_1^2s_i-s_j}{s_i-q_1^2s_j}
=-q_3^{l_0+1}q_1^{-l_0-2}
\frac{s_i-q_3^{-1}q_2^{-\mu}}{s_i-q_1^{-1}q_2^{-\mu}}
\frac{s_i-q_1q_3^{-2}}{s_i-q_3^{-1}}
\frac{a_1(q_2^{-1}q_1s_i)}{a_1(q_3s_i)}
\,.
\end{align*}

\subsection{Evaluation relaxed Verma modules}\label{sec:Wakimoto}
We extend our considerations in the previous subsection to the evaluation $\E^\perp$-module
obtained from the relaxed Verma module of $U_q(\widehat{\mathfrak{gl}}_2)$. 
We recall below the definition from \cite{JM}.

Let $\mu,\nu\in\C$, $\mu,\nu,\mu+\nu\not\in\Z$. 
Consider the $U_q\mathfrak{gl}_2$ module $L^{\mu,\nu}=\oplus_{i\in\Z}\C\, v_i$
given by
\begin{align*}
e_1 v_i=c_i v_{i-1}\,,\quad f_1 v_i=v_{i+1}\,, \quad
K v_i=q_3q^{-\mu-2\nu-2i}v_i\,,\quad \mathfrak{t}v_i=v_i\,,
\end{align*}
where 
\begin{align*}
c_i=\frac{(q_3q^{-\mu-\nu-i+1}-q_3^{-1}q^{\mu+\nu+i-1})
(q^{\nu+i}-q^{-\nu-i})}{(q-q^{-1})^2}\,.
\end{align*}
The module $L^{\mu,\nu}$ is characterized by the eigenvalue $q^{-\mu-2\nu}$ of $K$ on $v_0$,
and the value of the Casimir element 
$\mathcal{C}=q_3q^{-\mu+1}+q_3^{-1}q^{\mu-1}$. 

Let $U^+$ be the subalgebra of $U=U_q(\widehat{\mathfrak{gl}}_2)$ generated by
$x^\pm_{k}$, $h_{i,r}$ with $k\ge0$, $r>0$, and $K^{\pm1}$, $\mathfrak{t}^{\pm1}$, 
$C^{\pm1}$. Extending the action of $U_q\mathfrak{gl}_2$ on $L^{\mu,\nu}$ to $U^+$ by
\begin{align*}
x^{\pm}_k=h_{i,k}=0\,\quad \text{for $k>0$}\,, \quad C=q_3\,,
\end{align*} 
we define the induced $U$-module 
$\hat{L}^{\mu,\nu}=\mathrm{Ind}^U_{U^+}L^{\mu,\nu}$
of level $q_3$.
Define further $G_1^{\mu,\nu}$ to be the $\E^\perp$ module obtained from $\hat{L}^{\mu,\nu}$
via the evaluation map $ev_u$.  We call $G_1^{\mu,\nu}$
the evaluation relaxed Verma module. 
Unlike $G_1^\mu$, the module $G_1^{\mu,\nu}$ is not highest weight.

The module $G_1^{\mu,\nu}$ has a $\Z^2$-grading 
$G_1^{\mu,\nu}=\oplus_{l,d\in\Z}(G_1^{\mu,\nu})_{l,d}$ such that 
$(G_1^{\mu,\nu})_{l,d}=0$ for $d>0$, and its graded components have dimensions given 
as follows, see Figure \ref{flat pic} below:
\begin{align*}
\sum_{l,d\in\Z}\dim(G^{\mu,\nu}_1)_{l,d}\ z^l t^{-d} 
=\frac{1}{\prod_{j=1}^\infty(1-t^j)^4}\, \delta(z)\,.
\end{align*}
By construction, there is an isomorphism $G_1^{\mu,\nu}\simeq G_1^{\mu,\nu+1}$
of $\E^\perp$-modules 
which sends $(G_1^{\mu,\nu})_{l-1,d}$ to $(G_1^{\mu,\nu+1})_{l,d}$.

\vskip1cm

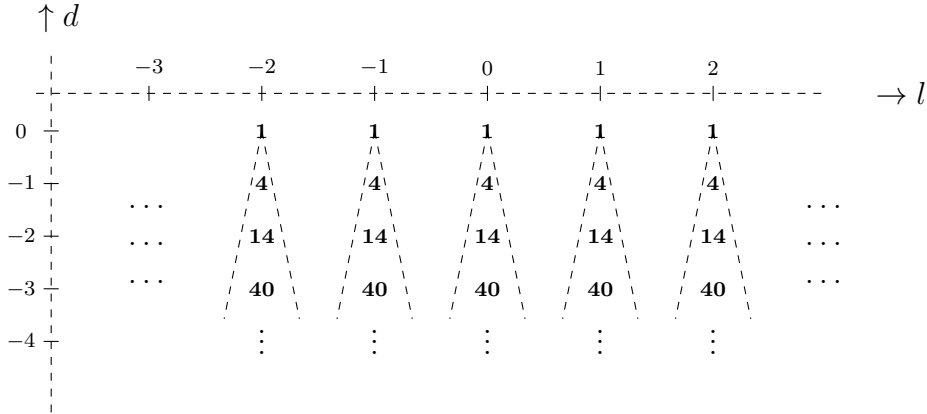
\begin{figure}[ht]
\centering
\begin{tikzpicture}
\draw[dashed] (-6,0.5) -- (4.5,0.5);
\node at (5.5,0.5) {$\to l$};
\draw[dashed] (-5.8,1.0) -- (-5.8,-3.8);
\node at (-5.7,1.5) {$\uparrow d$};
\coordinate (a1) at (-4.5,0);
\coordinate (a2) at (-3,0);
\coordinate (a3) at (-1.5,0);
\coordinate (a4) at (0,0);
\coordinate (a5) at (1.5,0);
\coordinate (a6) at (3,0);

\node[above] at ($(a1)+(0,0.6)$) {\tiny $-3$}; 
\node[above] at ($(a2)+(0,0.6)$) {\tiny $-2$};
\node[above] at ($(a3)+(0,0.6)$) {\tiny $-1$};
\node[above] at ($(a4)+(0,0.6)$) {\tiny $0$};
\node[above] at ($(a5)+(0,0.6)$){\tiny $1$};
\node[above] at ($(a6)+(0,0.6)$) {\tiny $2$};

\node at (-6.2,0) {\tiny $0$};
\draw (-5.9,0) -- (-5.7,0);
\node at (-6.2,-0.7) {\tiny $-1$};
\draw (-5.9,-0.7) -- (-5.7,-0.7);
\node at (-6.2,-1.4) {\tiny $-2$};
\draw (-5.9,-1.4) -- (-5.7,-1.4);
\node at (-6.2,-2.1) {\tiny $-3$};
\draw (-5.9,-2.1) -- (-5.7,-2.1);
\node at (-6.2,-2.8) {\tiny $-4$};
\draw (-5.9,-2.8) -- (-5.7,-2.8);

\draw ($(a1)+(0,0.4)$) -- ($(a1)+(0,0.6)$);
\draw ($(a2)+(0,0.4)$) -- ($(a2)+(0,0.6)$);
\draw ($(a3)+(0,0.4)$) -- ($(a3)+(0,0.6)$);
\draw ($(a4)+(0,0.4)$) -- ($(a4)+(0,0.6)$);
\draw ($(a5)+(0,0.4)$) -- ($(a5)+(0,0.6)$);
\draw ($(a6)+(0,0.4)$) -- ($(a6)+(0,0.6)$);

\draw[dashed] (a2) --++ (0.5,-2.5);
\draw[dashed] (a2) --++ (-0.5,-2.5);
\draw[dashed] (a3) --++ (0.5,-2.5);
\draw[dashed] (a3) --++ (-0.5,-2.5);
\draw[dashed] (a4) --++ (0.5,-2.5);
\draw[dashed] (a4) --++ (-0.5,-2.5);
\draw[dashed] (a5) --++ (0.5,-2.5);
\draw[dashed] (a5) --++ (-0.5,-2.5);
\draw[dashed] (a6) --++ (0.5,-2.5);
\draw[dashed] (a6) --++ (-0.5,-2.5);

\node at ($(a2)-(0,0)$) {\tiny${\bf 1}$};
\node at ($(a2)-(0,0.7)$) {\tiny${\bf 4}$};
\node at ($(a2)-(0,1.4)$) {\tiny${\bf 14}$};
\node at ($(a2)-(0,2.1)$) {\tiny${\bf 40}$};
\node at ($(a2)-(0,2.7)$) {$\vdots$};

\node at ($(a3)-(0,0)$) {\tiny${\bf 1}$};
\node at ($(a3)-(0,0.7)$) {\tiny${\bf 4}$};
\node at ($(a3)-(0,1.4)$) {\tiny${\bf 14}$};
\node at ($(a3)-(0,2.1)$) {\tiny${\bf 40}$};
\node at ($(a3)-(0,2.7)$) {$\vdots$};

\node at ($(a4)-(0,0)$) {\tiny${\bf 1}$};
\node at ($(a4)-(0,0.7)$) {\tiny${\bf 4}$};
\node at ($(a4)-(0,1.4)$) {\tiny${\bf 14}$};
\node at ($(a4)-(0,2.1)$) {\tiny${\bf 40}$};
\node at ($(a4)-(0,2.7)$) {$\vdots$};

\node at ($(a5)-(0,0)$) {\tiny${\bf 1}$};
\node at ($(a5)-(0,0.7)$) {\tiny${\bf 4}$};
\node at ($(a5)-(0,1.4)$) {\tiny${\bf 14}$};
\node at ($(a5)-(0,2.1)$) {\tiny${\bf 40}$};
\node at ($(a5)-(0,2.7)$) {$\vdots$};

\node at ($(a6)-(0,0)$) {\tiny${\bf 1}$};
\node at ($(a6)-(0,0.7)$) {\tiny${\bf 4}$};
\node at ($(a6)-(0,1.4)$) {\tiny${\bf 14}$};
\node at ($(a6)-(0,2.1)$) {\tiny${\bf 40}$};
\node at ($(a6)-(0,2.7)$) {$\vdots$};

\node at (4.5,-1) {$\dots$};
\node at (4.5,-1.5) {$\dots$}; 
\node at (4.5,-2) {$\dots$}; 

\node at (-4.5,-1) {$\dots$};
\node at (-4.5,-1.5) {$\dots$}; 
\node at (-4.5,-2) {$\dots$};

\end{tikzpicture}
\caption{Grading in the evaluation relaxed Verma module $G_1^{\mu,\nu}$}
\label{flat pic}
\end{figure}

The Bethe algebra $\mathcal{A}(p_0,p_1)$
acts on each graded component $(G_1^{\mu,\nu})_{l,d}$. 
Without loss of generality we may restrict our attention to the case $l=0$. 

First let $d=0$. As mentioned earlier, the action of
 $\mathcal{A}(p_0,p_1)$ on  $(G_1^{\mu,\nu})_{0,0}$ 
is independent of $p_0$
and coincides with that of the Bethe algebra of $U_q(\widehat{\mathfrak{gl}}_2)$. 
The $q$-oper corresponding to the latter on  $(G_1^{\mu,\nu})_{0,0}$ can 
be calculated, either by using the known 
formula for the universal $R$ matrix of $U_q\widehat{\mathfrak{sl}}_2$, 
or by analytic continuation in the parameters:
\begin{align}
L^{top}&=p_1q_2^{-\nu+1}(x-q^{\lambda_1-2}x_1)\sigma_{q_2}^2
-\bigl((p_1q_2+q_2^{-\nu+1}q^{\lambda_1})x-(p_1q^{\lambda_1}+1)x_1\bigr)
\sigma_{q_2}+q_2(q^{\lambda_1}x-q_2^{-1}x_1)
\label{Ltop}\\
&=q_2^{-\nu+1}x(p_1\sigma_{q_2}-q^{\lambda_1})(\sigma_{q_2}-q_2^{\nu})
-x_1(p_1q_2^{-\nu}q^{\lambda_1}\sigma_{q_2}-1)(\sigma_{q_2}-1)\,,
\nn
\end{align}
where
\begin{align}
q^{\lambda_1}=q_3q^{-\mu}\,,\quad x_1=q_3^{-1}q^{-\mu}\,.
\label{lax1}
\end{align}
This is a $q$-hypergeometric oper depending on three independent 
parameters $q^\mu,q^\nu$ and $p_1$.

For general $d$, we observe that 
 $G_1^{\mu,\nu}$ is an 
``analytic continuation'' of $G_1^\mu$ in the following sense.
The action of each generator of $\E^\perp$ on 
$(G_1^{\mu})_{l,d}$ for $l\ll 0$ depends on $l$ only through $q^l$.
Replacing  $q^{l}$ by $q^{-\nu}$ we obtain the action on
$(G_1^{\mu,\nu})_{0,d}$.

This suggests that the corresponding 
 $q$-opers should be obtained from the 
formulas in the previous section by the substitution $l_1=l_0+\nu$.

Let us state this as a conjecture. 
\begin{conj}\label{conj:1}
Assume that the parameters 
$q,q_3,\mu,\nu, \pz{0},\pz{1}$ are generic, and
let $\lambda_1,x_1$ be as in \eqref{lax1}. 
Then for each $l_0\in\Z_{\ge0}$
there exists a one-to-one correspondence between
joint eigenvectors of $\cA(\pz{0},\pz{1})$
on $(G_1^{\mu,\nu})_{0,-l_0}$
and the set of opers of the form 
\begin{align}
&L=a_0(x)\sigma_{q_2}^2-a_1(x)\sigma_{q_2}+a_2(x)\,,
\label{oper-relax}
\\
&a_0(x)=p_1q_2^{-\nu+1}(x-q^{\lambda_1-2}x_1)
\prod_{j=1}^{l_0}(x-q_1s_j)(x-q_3s_j)\,,
\nn\\ 
&a_2(x)=q_2^{l_0+1}(q^{\lambda_1}x-q^{-2}x_1)
\prod_{j=1}^{l_0}(x-q_2^{-1}q_1s_j)(x-q_2^{-1}q_3s_j)\,,
\nn\\
&a_1(x)=(p_1q_2^{l_0+1}+q_2^{-\nu+1}q^{\lambda_1})x^{2l_0+1}\nn\\
&+\sum_{j=1}^{2l_0}a_{1,j}x^{2l_0+1-j}
-(p_1q_2^{-l_0-\nu}q^{\lambda_1}+q_2^{-2l_0})
x_1 \prod_{j=1}^{l_0} s_j^2\,,
\nn
\end{align}
satisfying the conditions 
\begin{align}
& \Delta_1(q_1s_j) =\Delta_1(q_3s_j) =0\,,
\qquad 1\le j\le l_0\,,
\label{apparent-relax}\\
&p_0p_1\prod_{j=1}^{l_0}\frac{q_3^2s_i-s_j}{s_i-q_3^2s_j}
=-q_1^{l_0}q_3^{-l_0-1}
\frac{a_1(q_2^{-1}q_3s_i)}{a_1(q_1s_i)}\,,
\quad 1\le i\le l_0\,,
\label{BE0-relax}
\end{align}
where $\Delta_1(x)$ is given by \eqref{Del13}.
\end{conj}

The oper \eqref{oper-relax} is obtained from the top oper
\eqref{Ltop} by adjoining $2l_0$ apparent singularities $q_1s_j,q_3s_j$. 
The $3l_0$ unknowns 
$\{a_{1,j}\in\C\mid 1\le j\le 2l_0\}\cup\{s_j\in\C^{\times}\mid 1\le j\le l_0\}$
are to be determined from the $3l_0$ algebraic equations
\eqref{apparent-relax}, 
\eqref{BE0-relax}.

We remark that one of \eqref{apparent-relax}, e.g. $\Delta_1(q_3s_j)=0$, 
can be replaced by
\begin{align*}
&p_0p_1\prod_{j=1}^{l_0}\frac{q_1^2s_i-s_j}{s_i-q_1^2s_j}
=-q_3^{l_0+1}q_1^{-l_0-2}
\frac{s_i-q_3^{-1}q_2^{-\mu}}{s_i-q_1^{-1}q_2^{-\mu}}
\frac{s_i-q_1q_3^{-2}}{s_i-q_3^{-1}}
\frac{a_1(q_2^{-1}q_1s_i)}{a_1(q_3s_i)}
\,.
\end{align*}
These equations and \eqref{BE0-relax} are linear in the coefficients $a_{1,j}$, 
which can be solved in terms of $s_i$'s.
Using this, 
we have checked that the number of such opers for $l_0=1$ is $4$, 
in agreement with the character of $G_1^{\mu,\nu}$.

\bigskip

In \cite{JM} we considered also evaluation slanted relaxed Verma modules.
These are modules obtained by twisting $G_1^{\mu,\nu}$ by the 
automorphism 
\begin{align*}
&E_0^\perp(z)\mapsto z^{-m}E_0^\perp(z)\,,\quad 
F_0^\perp(z)\mapsto z^{m}F_0^\perp(z)\,,\\
&E_1^\perp(z)\mapsto z^{m}E_1^\perp(z)\,,\quad 
F_1^\perp(z)\mapsto z^{-m}F_1^\perp(z)\,,\\
&K^{\pm,\perp}_i(z)\mapsto K^{\pm,\perp}_i(z)\,,
\end{align*}
where $m\in \Z$ is called slope. 
It is natural to ask how the spectrum is described in terms of $q$-opers
for such modules. Except in the cases $m=0,1$ we do not know the answer. 

\section{Higher rank}\label{sec:higher}

In this section we examine
the possibility of generalizing the content of section \ref{sec:Main}
to quantum toroidal $\mathfrak{gl}_n$ algebras, which we denote
by $\E^{\perp}_n$, with $n\ge3$.

As in the $\mathfrak{gl}_2$ case, we expect that 
the Bethe ansatz method works for highest weight $\E^\perp_n$-modules with generic parameters.   
Let $\Psi(x)=(\Psi_0(x)\,,\ldots,\Psi_{n-1}(x))$ be the highest $\ell$-weight 
of the module considered, 
where $\Psi_i(x)$ are some rational functions such that $\Psi_i(0)\Psi_i(\infty)=1$. 
The expected Bethe ansatz equations read as 
\begin{align*}
(BAE)_\nu:\quad
&p_\nu q_2^{-l_\nu+(l_{\nu-1}+l_{\nu+1})/2}
\frac{y_\nu(q_2s^{(\nu)}_i)}{y_\nu(q_2^{-1}s^{(\nu)}_i)}
\frac{y_{\nu-1}(q_3s^{(\nu)}_i)}{y_{\nu-1}(q_1^{-1}s^{(\nu)}_i)}
\frac{y_{\nu+1}(q_1s^{(\nu)}_i)}{y_{\nu+1}(q_3^{-1}s^{(\nu)}_i)}
=-\Psi_\nu(s^{(\nu)}_i)\,,
\\
&\nu\in\Z/n\Z\,,\quad 1\le i\le l_\nu\,,
\end{align*}
where $p_\nu$'s are twist parameters entering the transfer matrix, and
$y_\nu(x)=\prod_{i=1}^{l_\nu}(x-s^{(\nu)}_i)$.
We will assume that $s^{(\mu)}_i/s^{(\nu)}_j\not\in q_2^{\Z}$ unless
$(\mu,i)=(\nu,j)$,
and that the $s^{(\nu)}_i$'s are not zeroes nor poles of 
$\Psi_\mu(q_1^iq_2^jx)$, $i,j\in\Z$, $\mu\in \Z/n\Z$.

If we regard $s^{(0)}_j$'s as given parameters, then  
(BAE)$_\nu$ with $1\le \nu\le n-1$ have the form of  
the Bethe equations for $U_q\widehat{\mathfrak{gl}}_n$. 
Following the procedure explained in \cite{MV}, 
we associate with them a $q$-difference oper of order $n$ (called a Miura $q$-oper)
\begin{align*}
L=a_0(x) (\sigma_{q_2}-b_n(x))\cdots(\sigma_{q_2}-b_2(x))(\sigma_{q_2}-b_1(x))\,.
\end{align*}
The leading coefficient $a_0(x)$ will be chosen later. 
Set
\begin{align*}
&\tilde{\Psi}_\nu(x)=(p_\nu q_2^{-l_\nu+(l_{\nu-1}+l_{\nu+1})/2})^{-1}{\Psi}_\nu(x)
=\frac{G_\nu(x)} {F_\nu(x)}
\,,
\end{align*}
where $F_\nu(x),G_\nu(x)$ are relatively prime polynomials, 
$F_\nu(x)$ being monic.
The $b_i(x)$'s are given by 
\begin{align}
&\frac{b_{i+1}(q_2^{-1}x)}{b_i(x)} 
=\tilde{\Psi}_i(q_3^{-i+1}x)
\frac{y_{i-1}(q_3^{-i+1}q_1^{-1}x)}{y_{i-1}(q_3^{-i+2}x)}
\frac{y_i(q_3^{-i+1}q_2^{-1}x)}{y_i(q_3^{-i+1}q_2x)}
\frac{y_{i+1}(q_3^{-i}x)}{y_{i+1}(q_3^{-i+1}q_1x)}\,,
\quad 1\le i\le n-1\,,
\label{b-rec}
\\
&b_1(x)=\frac{y_1(q^2x)}{y_1(x)}\,.
\nn
\end{align}
Explicitly we have
\begin{align*}
&b_i(x)=\prod_{\nu=1}^{i-1}\tilde{\Psi}_{\nu}(q_1^{\nu-1}q_2^{i-1}x)
\cdot
\frac{y_{i-1}(q_3^{-i+2}x)}{y_{i-1}(q_3^{-i+2}q_2x)}
\frac{y_i(q_3^{-i+1}q_2x)}{y_i(q_3^{-i+1}x)}
\frac{y_0(q_3 q_2^i x)}{y_0(q_3 q_2^{i-1} x)}\,
\quad (2\le i\le n-1),\\
&b_n(x)=\prod_{\nu=1}^{n-1}
\tilde{\Psi}_{\nu}(q_1^{\nu-1}q_2^{n-1}x)
\cdot
\frac{y_{n-1}(q_3^{-n+2}x)}{y_{n-1}(q_3^{-n+2}q_2x)}
\frac{y_0(q_3^{-n+1}q_2x)}{y_0(q_3^{-n+1}x)}
\frac{y_0(q_3 q_2^n x)}{y_0(q_3 q_2^{n-1} x)}
\,.
\end{align*}
If we write 
\begin{align}
&L=\sum_{k=0}^n(-1)^{k}a_k(x)\sigma_{q_2}^{n-k}
=a_0(x)\sum_{k=0}^n(-1)^kt_k(x)\sigma_{q_2}^{n-k}\,, 
\qquad t_k(x)=\frac{a_k(x)}{a_0(x)}\,,
\label{oper-rankn}
\end{align}
the coefficients $t_k(x)$ are given in terms of $b_i(x)$'s as  
\begin{align}
&t_k(x)=
\sum_{1\le i_1<\cdots<i_k\le n}
b_{i_k}(q_2^{n-i_k}x)
\cdots 
b_{i_p}(q_2^{n-i_p-k+p}x)
\cdots
b_{i_1}(q_2^{n-i_1-k+1}x)
\,.
\label{tkx}
\end{align}

We can choose $a_0(x),a_n(x)$ as 
\begin{align}
&a_0(x)=\prod_{1\le \nu<i\le n}F_\nu(q_1^{\nu-1}q_2^{i-1}x)
\cdot
\prod_{j=1}^{n-1}y_0(q_3q_2^jx)
\cdot y_0(q_3^{-n+1}x)
\,,\label{a0-rankn}
\\
&a_n(x)=\prod_{1\le \nu<i\le n}G_\nu(q_1^{\nu-1}q_2^{i-1}x)
\cdot
\prod_{j=1}^{n-1}y_0(q_3q_2^{j+1}x)
\cdot  y_0(q_3^{-n+1}q_2x)
\,.\label{bottom-rankn}
\end{align}

\begin{prop}
With the definition \eqref{a0-rankn}, \eqref{bottom-rankn} and 
under the Bethe equations (BAE)$_\nu$ with $1\le \nu\le n-1$, 
the coefficients $a_k(x)$ are polynomials for all $0\le k\le n$. 
\end{prop}
\begin{proof}
Since
the denominators of $t_k(x)$ arising from $\tilde{\Psi}_\nu(x)$ or $y_0(x)$
are canceled by $a_0(x)$, it is enough to check that 
$a_0(x)t_k(x)$ has no pole arising from $y_i(x)$ with $1\le i\le n-1$. 

Recall the definition \eqref{b-rec}. 
Due to the Bethe equations, the ratio $b_{i+1}(q_2^{-1}x)/b_i(x)$
takes the value $-1$
at $x=q_3^{i-1}s^{(i)}_a$. In other words $b_i(x)+b_{i+1}(q_2^{-1}x)$ 
does not have a pole at $y_i(q_3^{-i+1}x)=0$. 
Note also that, in the product $b_i(x)b_{i+1}(x)$, 
the factors containing $y_i(x)$ cancel out.

The assertion follows from these facts and \eqref{tkx}. 
\end{proof}
\medskip

\begin{prop}\label{prop:apparent-rankn}
The points 
\begin{align}
&q_3^{-1}q_2^{-j}s^{(0)}_a\,,\quad 1\le j\le n\,,\ 1\le a\le l_0\,,
\label{app1}\\
&q_3^{n-1}q_2^{-j}s^{(0)}_a\,,\quad j=0,1\,,\ 1\le a\le l_0\,,
\label{app2}
\end{align}
are apparent singularities of ${L}$.
\end{prop}
\begin{proof}
First consider the case \eqref{app1}. 
We apply Proposition \ref{prop:apparent-special} 
with $\ssq=q_2$
and $s=q_3^{-1}q_2^{-1}s^{(0)}_a$.

By the definition \eqref{b-rec}, $b_i(q_2^{n-i}x)$ with $i\neq 1$
and $b_n(x)$ share the same 
zeroes and poles of the type $q_2^js$.
By using this and \eqref{tkx},
the conditions (i),(ii) of Proposition \ref{prop:apparent-special} 
can be verified easily. 
Also the residue of $t_{k+1}(q_2^{-n+k+1}x)$ at $x=s$ is given by
\begin{align*}
\mathop{\mathrm{Res}}_{x=s} \sum_{2\le i_2<\cdots<i_{k+1}\le n}
b_{i_{k+1}}(q_2^{-i_{k+1}+k+1}x)\cdots b_{i_2}(q_2^{-i_2+1}x)b_1(x)\,.
\end{align*}
Since $t_1(q_2^{-n+1}s)=b_1(s)$, this coincides with 
the residue of $t_k(q_2^{-n+k+1}x)t_1(q_2^{-n+1}x)$ at $x=s$.
Hence the condition
(iii) is also satisfied.

It remains to check \eqref{app2}.
We are to show that 
$A(x)A(q_2^{-1}x)$ is regular at $x=q_3^{n-1}s^{(0)}_a$, 
or equivalently that
\begin{align*}
t_k^{(1)}(x)=
\left|
\begin{matrix}
 t_1(x) & t_{k+1}(x)\\
1 & t_k(q_2^{-1}x)\\
\end{matrix}
\right|\,,\quad 1\le k\le n-1\,,
\end{align*}
are regular there. This can be shown by noting that
$b_n(x)$ has a pole, $b_n(q_2^{-1}x)$ has a zero, and $b_l(x)$ with $l\neq n$
are regular at $x=q_3^{n-1}s^{(0)}_a$. 
\end{proof}

\begin{prop}\label{prop:BAE0}
The Bethe ansatz equations (BAE)$_0$ can be rewritten as
\begin{align}
&\frac{a_0(q_3^{-1}q_2^{-n}s^{(0)}_i)} {a_1(q_3^{-1}q_2^{-n}s^{(0)}_i)} 
\frac{a_0(q_3^{n-1}q_2^{-1}s^{(0)}_i)}{a_{n-1}(q_3^{n-1}q_2^{-1}s^{(0)}_i)}
\frac{y_0(q_2s^{(0)}_i)}{y_0(q_2^{-1}s^{(0)}_i)}
\frac{y_0(q_1^{-n}q_2^{-1}s^{(0)}_i)}{y_0(q_3^{n}q_2 s^{(0)}_i)}
\label{BAE0-subst}
\\
&=
-\tilde{\Psi}_0(s^{(0)}_i)
\prod_{1\le j< k\le n-1}
\tilde{\Psi}_j(q_3^{n-j}q_2^{k-j}s^{(0)}_i)^{-1}\,,
\qquad 1\le i\le l_0\,.\nn
\end{align}
\end{prop}
\begin{proof}
We must eliminate $y_1(x)$ and $y_{n-1}(x)$ from the left
hand side of (BAE)$_0$.
For that purpose we use
\begin{align*}
&t_1(x)=\sum_{i=1}^{n}b_i(q_2^{n-i}x)\,,\quad
t_{n-1}(x)=\sum_{i=1}^n \prod_{j=i+1}^n b_j(x)\cdot
\prod_{j=1}^{i-1}b_j(q_2x)\,. 
\end{align*} 
From the explicit formula for $b_j(x)$ we see that
$b_j(q_2^{n-j}x)$ with $j\neq 1$
have a common zero at $x=q_3^{-1}q_2^{-n}s^{(0)}_i$. Hence
\begin{align*}
t_1(q_3^{-1}q_2^{-n}s^{(0)}_i) 
=b_1(q_1 s^{(0)}_i) =\frac{y_1(q_3^{-1}s^{(0)}_i)}{y_1(q_1s^{(0)}_i)}\,.
\end{align*}
Similarly, noting that $b_n(q_3^{n-1}q_2^{-1}s^{(0)}_i)=0$
and that for $p\le n-1$
\begin{align*}
\prod_{i=1}^p b_i(x)=\prod_{1\le j<k\le p} 
\tilde{\Psi}_j(q_1^{j-1}q_2^{k-1}x)
\cdot
\frac{y_p(q_3^{-p+1}q_2x)}{y_p(q_3^{-p+1}x)}
\frac{y_0(q_1^{-1}q_2^{p-1}x)}{y_0(q_1^{-1}x)}\,,
\end{align*}
we obtain
\begin{align*}
t_{n-1} (q_3^{n-1}q_2^{-1}s^{(0)}_i)&=
\prod_{k=1}^{n-1}b_k(q_3^{n-1}s^{(0)}_i)
\\
&=\prod_{1\le j<k\le n-1}
\tilde{\Psi}_j(q_1^{-n+j}q_2^{-n+k}s^{(0)}_i)
\cdot
\frac{y_{n-1}(q_1^{-1}s^{(0)}_i)}{y_{n-1}(q_3s^{(0)}_i)}
\frac{y_0(q_1^{-n}q_2^{-1}s^{(0)}_i)}{y_0(q_3^{n}q_2 s^{(0)}_i)}\,.
\end{align*}
Substituting these into (BAE)$_0$, we obtain the stated formula.
\end{proof}

We now consider a $q$-difference oper \eqref{oper-rankn}  
with \eqref{a0-rankn}, \eqref{bottom-rankn}, regarding  
\begin{align}
q_2^{l_1}\,,\ \ldots,\ q_2^{l_{n-1}}\,
\label{qli}
\end{align}
as independent complex parameters.
We impose the condition that the zeroes of $a_0(x)$ coming from $y_0(x)$
are all apparent singularities as stated in 
Proposition \ref{prop:apparent-rankn} 
(more precisely we demand the conditions in 
Proposition \ref{prop:apparent-special}), 
as well as the condition \eqref{BAE0-subst}. 
The issue is whether these conditions determine $L$ to within a finite number.

For that matter let us count the number of equations and that of unknowns. 
We have 
\begin{align*}
\deg a_0(x)=\deg a_{n}(x)=n l_0+\sum_{j=1}^{n-1}(n-j)\deg F_j(x)\,. 
\end{align*}
The polynomials $a_1(x),\ldots,a_{n-1}(x)$ have the same degree.
Their leading coefficients at $x\to\infty$ and constant terms are 
determined by \eqref{qli} via
\begin{align*}
&\sum_{k=0}^n(-1)^k t_k(\infty)\sigma_{q_2}^{n-k}
=\prod_{j=1}^n (\sigma_{q_2}-b_k(\infty))\,,\quad
\sum_{k=0}^n(-1)^k t_k(0)\sigma_{q_2}^{n-k}
=\prod_{j=1}^n (\sigma_{q_2}-b_k(0))\,.
\end{align*}
Along with the roots $s^{(0)}_i$ themselves, there are altogether
\begin{align}
(n-1)\bigl(n l_0+\sum_{j=1}^{n-1}(n-j)\deg F_j(x)-1\bigr) +l_0
\label{number-of-unknowns}
\end{align} 
unknowns. 

For each  $s^{(0)}_i$, 
the condition that \eqref{app1}, \eqref{app2} are apparent
amounts to $(n-1)^2+(n-1)=n(n-1)$ conditions.
Together with the condition in Proposition \ref{prop:BAE0}, we find that 
the total number of equations is
\begin{align}
n(n-1)l_0+l_0\,. \label{number-of-eqs}
\end{align}
Then \eqref{number-of-unknowns} and \eqref{number-of-eqs}
match only when $\deg F_1(x)=\cdots=\deg F_{n-2}(x)=0$
and $\deg F_{n-1}(x)=1$.

The reasoning above suggests that an analog of 
Conjecture \ref{conj:1} for the quantum toroidal $\gl_n$ algebra
holds true if we consider
an evaluation  $\E^\perp_n$-module of highest $\ell$-weight 
\begin{align*}
\Bigl(
\frac{\kappa_0(1-\kappa_0^{-2}u/z)}{1-u/z},1,\ldots,1,
\frac{\kappa_{n-1}(1-q_3^{-1}u/z)}{1-q_3^{-1}\kappa_{n-1}^2u/z}
\Bigr)\,, 
\qquad \kappa_0^2\kappa_{n-1}^2=q_3^n\,,
\end{align*}
and ``analytically continue'' with respect to the parameters \eqref{qli}.
Alternatively one can start from a representation of $U_q\mathfrak{gl}_n$
with $n$ parameters $\kappa_{n-1},q_2^{l_1},\ldots,q_2^{l_{n-1}}$, 
induce it up to $U_q\widehat{\mathfrak{gl}}_n$ and consider 
the evaluation $\E_n^\perp$-module. We do not repeat the details here. 

We remark that, 
with such a choice of $\ell$-weight, 
the top oper (one for which $y_0(x)=1$) takes the form 
\begin{align*}
L^{top}&=
x(\hat{p}_{n}\sigma_{q_2}-\kappa_{n-1}q^{l_1-l_{n-1}})
\prod_{i=1}^{n-1}(\hat{p}_i\sigma_{q_2}-q^{l_1-l_{i-1}+l_i})
\\
&-x_1(\hat{p}_n\sigma_{q_2}-\kappa_{n-1}^{-1}q^{l_1+l_{n-1}})
\prod_{i=1}^{n-1}(\hat{p}_i\sigma_{q_2}-q^{l_1+l_{i-1}-l_i})\,,
\end{align*}
where $\hat{p}_i=\prod_{j=1}^{i-1}p_j$ and
$x_1=q_3^{-1}\kappa_{n-1}^2u$. 
After appropriate rescaling and conjugation,  
the equation $L^{top}y=0$ reduces to the one 
for the generalized $q$-hypergeometric series
\begin{align*}
{}_n\phi_{n-1}\Bigl(
\begin{matrix}
 a_1 & \cdots &a_{n-1}& a_n\\
b_1 & \cdots &b_{n-1}&
\end{matrix};
q_2,x
\Bigr) 
 =\sum_{j=0}^\infty\frac{(a_1,\ldots,a_{n-1},a_n; q_2)_j}
{(b_1,\ldots,b_{n-1},q_2;q_2)_j}x^j\,.
\end{align*}
This is a rigid equation \cite{R} 
(i.e. it does not have accessory parameters), 
and its connection matrix has been determined explicitly, 
see e.g. \cite{Wa}.

\appendix

\section{Differential limit of apparent singularities}\label{app:1}

In this section we examine the limit $\ssq\to 1$ 
of second order equations.
We show that the conditions for apparent singularities
imply the semi-simplicity of local monodromy 
(i.e. absence of a logarithmic term)
for the differential equations obtained as the limit.

\subsection{Second order operators}

We consider a second order $q$-difference operator with polynomial coefficients
\begin{align*}
L=a_0(x,\ssq)\sigma_\ssq^2-a_1(x,\ssq)\sigma_\ssq+a_2(x,\ssq)\,, 
\end{align*}
where the dependence on $\ssq$ is made explicit. 
Let $s\in \C^{\times}$ and $r\in\Z_{\ge0}$. 
As in Proposition \ref{prop:2nd}, we assume that 
$a_0(x,\ssq)$, $a_2(\ssq^{-r}x,\ssq)$ have a simple zero at $x=s$
and that there are no other zeros of the form $\ssq^j s$, $j\in\Z$. 
We set $\ssq=e^{-\epsilon}$, keep $s$ fixed 
and study the limit $\epsilon\to 0$.
We assume also that no other zeroes converge to $s$ in this limit. 

In view of the relation 
\begin{align*}
\frac{\sigma_\ssq-1}{\ssq-1}\longrightarrow x\partial 
\quad \text{as $\ssq\to 1$} 
\end{align*}
where $\partial=d/dx$, we rewrite $L$ in the form 
\begin{align*}
\frac{1}{(\ssq-1)^2}L
=
\tilde{a}_0(x,\ssq)\Bigl(\frac{\sigma_\ssq-1}{\ssq-1} \Bigr)^2
-\tilde{a}_1(x,\ssq)\frac{\sigma_\ssq-1}{\ssq-1} 
+\tilde{a}_2(x,\ssq)\,,
\end{align*}
and assume that $\tilde{a}_i(x,\ssq)$'s have well-defined limits as $\ssq\to 1$.
This amounts to imposing the condition
\begin{align}\label{expansion start}
&a_0(x,\ssq)-a_2(x,\ssq)=O(\epsilon)\,,\quad
a_0(x,\ssq)+a_2(x,\ssq)-a_1(x,\ssq)=O(\epsilon^2) \,.
\end{align}

The limiting differential operator reads
\begin{align}
\mathcal{L}=\lim_{\epsilon\to0}\epsilon^{-2}L
=\alpha(x)(x\partial)^2-\beta(x) x\partial+\gamma(x)\,,
\label{diff-op}
\end{align}
where the coefficients are polynomials given by
\begin{align*}
&\alpha(x)=\lim_{\epsilon\to 0}a_0(x,\ssq)\,, \\
&\beta(x)=\lim_{\epsilon\to0}\epsilon^{-1}\bigl(a_0(x,\ssq)-a_2(x,\ssq)\bigr)\,,\\
&\gamma(x)=\lim_{\epsilon\to 0}\epsilon^{-2}
\bigl(a_0(x,\ssq)+a_2(x,\ssq)-a_1(x,\ssq)\bigr)\,.
\end{align*}

The assumption about the zeroes of $a_0(x,\ssq)$ and $a_2(\ssq^{-r}x,\ssq)$ 
imply that
\begin{align}
\alpha(s)=0\,,\quad  \beta(s)=r s (\partial\alpha)(s)\,,
\quad \partial \alpha(s)\neq0\,.
\label{reg-sing-cond}
\end{align}
Consequently 
the equation $\mathcal{L}y=0$ has a regular singularity at $x=s$ with
characteristic exponents $0,r+1$. 
Since the latter differ by a positive integer, 
the local monodromy has a Jordan block in general, 
and it is semi-simple if and only if there exists a local
holomorphic solution $y(x)=\sum_{j\ge0}c_j(x-s)^j$
satisfying $c_0\neq0$. 
Denote by 
 $\alpha_j,\beta_j,\gamma_j$ the Taylor coefficients at $x=s$
of $x^2\alpha(x)$, $x(\beta(x)-\alpha(x))$, $\gamma(x)$, respectively.
Then the coefficients of 
$\mathcal{L}y(x)=\sum_{j\ge0}\tilde{c}_j(x-s)^j$ are given by
\begin{align}
&\tilde{c}_i=\sum_{j=0}^{i+1} N_{i,j}c_j\,,\quad 0\le i\le r\,,
\label{Nyy}\\
&N_{i,j}=j(j-1)\alpha_{i+2-j}-j\beta_{i+1-j}+\gamma_{i-j}\,.
\nn
\end{align}
In particular, using \eqref{reg-sing-cond}, 
we have $N_{i,i+1}=(i+1)(i-r)\alpha_1$, so 
$c_{r+1}$ does not appear in 
the right hand side of $\tilde{c}_r$.

Therefore the local monodromy is semi-simple if and only if 
\begin{align*}
\det N=0\,\quad 
\text{ where $N=\bigl(N_{i,j}\bigr)_{0\le i,j\le r}$}\,.
\end{align*}

\begin{prop}\label{prop:limdelta}
For the determinant $\Delta_r(x)$ in \eqref{tridiag}, we have 
\begin{align*}
\lim_{\epsilon\to 0}\epsilon^{-2r-2}\Delta_r(s)=\det(-N)\,.
\end{align*}
Hence the condition $\Delta_r(s)=0$ for the apparent singularity 
implies the no-log condition $\det N=0$ in the limit.
\end{prop}
\begin{proof}
Define matrices $P=\bigl(P_{i,j}\bigr)_{0\le i,j\le r}$ ,
$M=\bigl(M_{i,j}\bigr)_{0\le i,j\le r}$
by
\begin{align*}
&P_{i,j}=(-1)^i \binom{i}{j}\epsilon^j\,,
\quad
M_{i,j}=\delta_{i,j+1}a_0(\ssq^{-i}s,\ssq)+\delta_{i,j}a_1(\ssq^{-i}s,\ssq)+
\delta_{i,j-1}a_2(\ssq^{-i}s,\ssq)\,,
\end{align*} 
where we set $\binom{i}{j}=0$ if $j>i$ or $j<0$.
We have then $\Delta_r(s)=\det M$. 

Setting
$M_{0,-1}=a_0(s,\ssq)=0$, $M_{r,r+1}=a_2(\ssq^{-r}s,\ssq)=0$ and noting that
$(P^{-1})_{i,j}=\epsilon^{-i}(-1)^{i}\binom{i}{j}$,
we compute
\begin{align}
(P^{-1}MP)_{i,j}&=\sum_{k=0}^r (P^{-1})_{i,k}\Bigl(a_0(\ssq^{-k}s,\ssq)P_{k-1,j}
+a_1(\ssq^{-k}s,\ssq)P_{k,j}+a_2(\ssq^{-k}s,\ssq)P_{k+1,j}\Bigr) 
\label{PMP1}
\\
&=\epsilon^{-i+j}\sum_{k=0}^r(-1)^{i-k}\binom{i}{k}
\Bigl\{-\binom{k-1}{j-2}a_0(\ssq^{-k}s)
\nn\\
&+\binom{k}{j}\bigl(a_1(\ssq^{-k}s)-a_0(\ssq^{-k}s)-a_2(\ssq^{-k}s)\bigr)
+\binom{k}{j-1}\bigl(a_0(\ssq^{-k}s)-a_2(\ssq^{-k}s)\bigr)\Bigr\}\,.\nn
\end{align}
For any function $f(x)$ which is holomorphic at $x=0$, we have
$f(\ssq^{-k}x)=\bigl(e^{k\epsilon x\partial}f\bigr)(x)$, so that
\begin{align*}
\sum_{k=0}^r(-1)^{i-k} \binom{i}{k} \binom{k}{j} f(\ssq^{-k}s)
&=\sum_{k=0}^r (-1)^{i-k}\binom{i}{j}\binom{i-j}{i-k}
\sum_{p\ge0}\frac{(k\epsilon)^p}{p!}\bigl((x\partial)^pf\bigr)(s)
\\
&=\binom{i}{j}\sum_{p\ge0}\frac{\epsilon^{p}}{p!}
\bigl((x\partial)^p f\bigr)(s)
\sum_{l=0}^{i-j}(-1)^l(i-l)^p
\binom{i-j}{l}\\
&=\binom{i}{j}\sum_{p\ge0}\frac{\epsilon^{p}}{p!}
\bigl((x\partial)^p f\bigr)(s)\cdot
(i-x\partial)^p(1-x)^{i-j}\Bigl|_{x=1}\,
\\
&=\binom{i}{j}\epsilon^{i-j}\bigl((x\partial)^{i-j} f\bigr)(s)
+O(\epsilon^{i-j+1})\,.
\end{align*}
Applying this identity to the second and third terms of the second line of
\eqref{PMP1}, 
and computing similarly for the first term,
we find that 
\begin{align}
&(P^{-1}MP)_{i,j}=\epsilon^2 \tilde{M}_{i,j}+O(\epsilon^3)\,,
\label{PMP}
\end{align}
where
\begin{align*}
\tilde{M}_{i,j}
&=
-\binom{i}{j-2}\bigl((x\partial)^{i-j+2}\alpha\bigr)(s)
+\binom{i}{j-1}\bigl((x\partial)^{i-j+1}\beta\bigr)(s)
-\binom{i}{j}\bigl((x\partial)^{i-j}\gamma\bigr)(s)
\,.
\end{align*}

For a power series $y(x)=\sum_{j\ge0}c_j(x-s)^j$, let 
$\tilde{y}(x)=\mathcal{L}y(x)=\sum_{j\ge0}\tilde{c}_j(x-s)^j$. 
Due to the Leibniz rule, the following relation holds
for any $0\le i\le r$:
\begin{align}
\bigl((x\partial)^i\tilde{y} \bigr)(s)
=-\sum_{j=0}^r\tilde{M}_{i,j}
\bigl((x\partial)^jy\bigr)(s)\,.
\label{Myy}
\end{align}
In the right hand side, terms $(x\partial)^jy$ with $j>r$ drop 
due to \eqref{reg-sing-cond}.

On the other hand, we have clearly
\begin{align*}
\bigl((x\partial)^i\tilde{y} \bigr)(s)=\sum_{j=0}^r K_{i,j}\tilde{c}_j\,,
\quad
\bigl((x\partial)^i y \bigr)(s)=\sum_{j=0}^r K_{i,j}c_j\,,
\end{align*}
for some triangular matrix $K$. 
Comparing \eqref{Myy} with \eqref{Nyy} and using \eqref{PMP}, 
we conclude that 
\begin{align*}
\det M= \epsilon^{2r+2}\det\tilde{M}+O(\epsilon^{2r+3})\,,
\quad 
\det \tilde{M}=\det \bigl(-N\bigr)\,.
\end{align*}
This completes the proof. 

\end{proof}
\bigskip

\noindent {\it Remark.}\quad
Often it is preferred to bring \eqref{diff-op}
into the normal form 
\begin{align*}
&U(x)\mathcal{L}U(x)^{-1}
=x^2\alpha(x)\bigl(\partial^2-V(x)\bigr)\,,
\end{align*}
by conjugation with appropriate $U(x)$. At $x=s$ 
\begin{align}
V(x)&=\frac{\frac{r}{2}\bigl(\frac{r}{2}+1\bigr)}{(x-s)^2}
+\frac{V_1}{x-s}+V_2+V_3(x-s)+\cdots\,,
\label{expandV}
\end{align}
so the characteristic exponents are $-r/2,r/2+1$. 
The criterion for semi-simplicity becomes the vanishing of the determinant
\begin{align}
N_r=
\left|
\begin{matrix}
V_1 &1\cdot r       &    0              &0 &\cdots&0\\
V_2& V_1     &2\cdot(r-1)            &0 &\cdots&0\\
V_3& V_2     &V_1    &3\cdot(r-2)    &\cdots&0 \\
 &&&&&\\
\vdots   & \ddots  &\ddots&\ddots  &\ddots& 0  \\
\vdots   & \ddots  &\ddots& V_2 &V_1&r\cdot 1   \\
V_{r+1}&\cdots&        & V_3&V_2  & V_1 \\
\end{matrix}
\right| \,.
\label{Nolog}
\end{align}
\qed
\bigskip

\subsection{Another example}

Next we consider the oper \eqref{oper-relax} 
relevant to Conjecture \ref{conj:1}. 
We choose the shift parameter to be $\ssq=q_2$. 
The coefficients $a_0(x,\ssq)$, $a_2(\ssq^{-1}x,\ssq)$ 
have a pair of zeros $x=q_1s,q_3s$, where $s$ stands for one of $s_i$'s.
Since  these zeroes merge in the limit $q_1,q_2,q_3\to 1$, 
Proposition \ref{prop:limdelta} does not apply, and a separate consideration is necessary.

We set 
\begin{align*}
q_2=e^{-\epsilon}, \quad q_1= e^{-t\epsilon}\,,
\quad q_3=e^{(1+t)\epsilon}\,,
\end{align*}
and let $\epsilon \to 0$, keeping $s$ and $t$ fixed.
We use the expansions
\begin{align*}
&a_0(x,\ssq)=(x-q_1s)(x-q_3s)\bigl(\ba_0(x)+\epsilon \ba_1(x)
+\epsilon^2 \ba_2(x)+\epsilon^3\ba_3(x)+\cdots\bigr) \,,
\\
&a_1(x,\ssq)=
2(x-s)^2\ba_0(x)+\epsilon (x-s)^2\bigl(\ba_1(x)+\bc_1(x)\bigr)
-4\epsilon(x-s)s \ba_0(x)
+\epsilon^2 \bb_2(x)+\epsilon^3 \bb_3(x)+\cdots \,,
\\
&a_2(x,\ssq)=(x-q_1q_2^{-1}s)(x-q_3q_2^{-1}s)\bigl(\ba_0(x)+\epsilon \bc_1(x)
+\epsilon^2 \bc_2(x)+\epsilon^3\bc_3(x)+\cdots\bigr) \,.
\end{align*}
Here the expansions are chosen so that \eqref{expansion start} is satisfied.
The condition for $s$ to be an apparent singularity is the vanishing
of two determinants
\begin{align*}
\Delta^{(1)}
=\left|
\begin{matrix}
a_1(q_1 s,\ssq)  & a_2(q_1 s,\ssq) \\
a_0(q_1q_2^{-1} s,\ssq)  & a_1(q_1q_2^{-1} s,\ssq) \\
\end{matrix}
\right|\,,
\quad
\Delta^{(2)}
=\left|
\begin{matrix}
a_1(q_3s,\ssq)  & a_2(q_3s,\ssq) \\
a_0(q_3q_2^{-1} s,\ssq)  & a_1(q_3q_2^{-1} s,\ssq) \\
\end{matrix}
\right|\,.
\end{align*}
We expand them in powers of $\epsilon$ as 
$\Delta^{(i)}=\sum_{n=0}^\infty \epsilon^n\Delta^{(i)}_n$. 
The first nontrivial order is $\epsilon^4$ which gives
\begin{align*}
\Delta^{(1)}_4=\Delta^{(2)}_4=
-\Bigl(\bb_2(s)+2s^2(-1+t+t^2)\ba_0(s)\Bigr) 
\Bigl(\bb_2(s)+2s^2(t+t^2)\ba_0(s)\Bigr)\,.
\end{align*}
There are two cases to consider,
\begin{align*}
\text{(1): $\bb_2(s)=-2s^2(-1+t+t^2)\ba_0(s)$},
\quad
\text{(2): $\bb_2(s)=-2s^2(t+t^2)\ba_0(s)$}.
\end{align*}
We will see that they correspond to the decomposition $\C^2\otimes \C^2=\C\oplus \C^3$.
Substituting these conditions respectively, 
we compute the limiting differential operator in the normal form
$\partial^2-V$. 

In the case (1), we obtain
\begin{align*}
V(x)=\frac{V_1}{x-s}+V_2+O\bigl(x-s\bigr) 
\end{align*}
which corresponds to $r=0$ in \eqref{expandV}.  
At the next order we get the no-log condition
\begin{align*}
\Delta^{(1)}_5-\Delta^{(2)}_5=-4s^5(1+2t) \ba_0(s)^2\times N_0\,
\end{align*}
in the notation of \eqref{Nolog}. 

In the case (2) we have
\begin{align*}
V(x)=\frac{2}{(x-s)^2}+\frac{V_1}{x-s}+V_2+V_3(x-s)+O\bigl((x-s)^2\bigr) 
\end{align*}
which corresponds to $r=2$. 
Proceeding further we find that
\begin{align*}
&\text{$\Delta^{(1)}_5=\Delta^{(2)}_5$ gives $\bb_3(s)$},\\
&\text{$\Delta^{(1)}_6=\Delta^{(2)}_6$ gives $\bb_4(s)$},
\end{align*}
in terms of differential polynomials of
$\ba_i(x)$, $0\le i\le 2$, $\bb_2(x)$ and $\bc_j(x)$, $1\le j\le 2$ , 
evaluated at $x=s$, up to a denominator of the form $\ba_0(s)^m$.

Substituting them back at the next order, we arrive at 
\begin{align*}
\Delta^{(1)}_7-\Delta^{(2)}_7
= -t(1+t)(1+2t)s^7 \ba_0(s)^2\times N_2\,.
\end{align*}
\bigskip

\section{Connection matrix for $q$-hypergeometric opers}
\label{subsec:connection}

In this section we study the connection matrix for $q$-opers obtained by 
adjoining apparent singularities to a $q$-hypergeometric oper.
We will work with the corresponding first order systems.

We fix $\xi,\rho_1,\rho_2,\kappa_1,\kappa_2\in\C^{\times}$, and 
a polynomial $p(x)=\prod_{i=1}^N(x-s_i)$ such that $s_i\in\C^\times$ and
$s_i/s_j\not \in \ssq^{\Z}$ for $i\neq j$. 
Consider a $q$-difference equation in the first order form
\begin{align}
 Y(\ssq x)=A(x)Y(x)\,,\quad
A(x)=\begin{pmatrix}
      \frac{a_1(x)}{a_0(x)}& -\frac{a_2(x)}{a_0(x)}\\
           1  & 0 \\
     \end{pmatrix}\,,
\label{qHGAY}
\end{align}
where
\begin{align*}
&a_0(x)=(x-\xi)p(x)\,,\\
&a_1(x)=(\kappa_1+\kappa_2)x^{N+1}+\sum_{k=1}^Na_{1,k}x^k
-\xi(\rho_1+\rho_2)p(0)\,,\\
&a_2(x)=\bigl(\kappa_1\kappa_2\ssq^{-N}x-\xi \rho_1\rho_2\bigr)p(\ssq x)\,.
\end{align*}
We assume the non-resonance condition 
\begin{align}
\frac{\rho_1}{\rho_2}\,,\frac{\kappa_1}{\kappa_2}
\quad \not\in \ssq^\Z\,,
\label{non-resonant}
\end{align}
and that the $s_i$ and $\ssq^{-1} s_i$ are apparent singularities.
Since $a_0(s_i)=a_2(\ssq^{-1}s_i)=0$, we have 
by Proposition \ref{prop:2nd}
\begin{align*}
\left|
 \begin{matrix}
  a_1(s_i) & a_2(s_i)\\
  a_0(\ssq^{-1}s_i)& a_1(\ssq^{-1}s_i)
 \end{matrix}
\right|
=0\,,
\quad 1\le i\le N\,.
\end{align*}

The coefficients $A_0,A_\infty$ are independent of $N$ and $s_i$,
\begin{align}
A_0=\begin{pmatrix}
      \rho_1+\rho_2 &-\rho_1\rho_2\\
       1 & 0 \\
     \end{pmatrix} \,,
\quad
A_\infty=\begin{pmatrix}
      \kappa_1+\kappa_2 &-\kappa_1\kappa_2\\ 1 & 0 \\
     \end{pmatrix} \,.
\label{A0Ainf}
\end{align}
It will be convenient to diagonalize them 
\begin{alignat*}{2}
&A_0=C_0\Lambda_0 C_0^{-1}\,,
\quad
&C_0=
\begin{pmatrix}
\rho_1 & \rho_2\\
1 & 1\\
\end{pmatrix}
\,,
\quad 
&\Lambda_0=
\begin{pmatrix}
\rho_1&0\\
0&\rho_2\\ 
\end{pmatrix} 
\,,
\\
&A_\infty=C_\infty\Lambda_\infty C_\infty^{-1}\,,
\quad
&C_\infty=
\begin{pmatrix}
\kappa_1 & \kappa_2\\
1 & 1\\
\end{pmatrix}
\,,
\quad 
&\Lambda_\infty=
\begin{pmatrix}
\kappa_1&0\\
0&\kappa_2\\ 
\end{pmatrix} 
\,,
\end{alignat*}
and modify the canonical solutions accordingly:
\begin{align*}
&Y_0'(x)=Y_0(x)C_0
=\widehat{Y}'_0(x) 
\begin{pmatrix}
 e_{\rho_1}(x)& 0 \\
0& e_{\rho_2(x)}\\
\end{pmatrix}\,,
\quad \widehat{Y}'_0(0) =C_0\,,
\\
&Y'_\infty(x)=Y_\infty(x)C_\infty
=\widehat{Y}'_\infty(x)
\begin{pmatrix}
 e_{\kappa_1}(x)& 0 \\
0& e_{\kappa_2(x)}\\
\end{pmatrix}\,,
\quad \widehat{Y}'_\infty(\infty)=C_\infty
\,.
\end{align*}
The function $e_c(x)$ is given in \eqref{ec}.

Solving the equation $\det Y(\ssq x)=\det A(x)\det Y(x)$, we find
\begin{align*}
&\det Y'_0(x)=e_{\rho_1}(x)e_{\rho_2}(x)
\frac{(x/\xi;\ssq)_\infty}{(\eta 
\ssq^{-N}x/\xi;\ssq)_\infty}\frac{p(x)}{p(0)}(\rho_1-\rho_2)
\,,
\\
&\det Y'_\infty(x)=e_{\kappa_1}(x)e_{\kappa_2}(x)
\frac{(\ssq^{N+1} \xi/(\eta x);\ssq)_\infty}{(\ssq\xi/x;\ssq)_\infty}
\frac{p(x)}{x^N} (\kappa_1-\kappa_2)
\,,
\end{align*}
where $\eta=\kappa_1\kappa_2/(\rho_1\rho_2)$. 
Set 
\begin{align*}
\hat{M}'(x)={\widehat{Y}_0'(x)}^{-1}\widehat{Y}'_\infty(x).
\end{align*}
We have then
\begin{align*}
&\det \hat{M}'(x)=K\frac{\theta(\eta x/\xi)}{\theta(x/\xi)}\,,
\quad K=\ssq^{-N(N+1)/2}\bigl(-\frac{\eta}{\xi}\bigr)^{N}
p(0)
\frac{\kappa_1-\kappa_2}{\rho_1-\rho_2}
\,,\\
&\text{the poles of $\hat{M}'(x)$ are contained in $\ssq^\Z\xi$},\\
& \hat{M}'(\ssq x)=\Lambda_0 \hat{M}'(x)\Lambda_{\infty}^{-1}\,.
\end{align*}
The last condition implies 
$(\hat{M}'(\ssq x))_{i,j}=(\hat{M}'(x))_{i,j}\rho_i\kappa_j^{-1}$.
Together with the pole condition we conclude that
\begin{align}
(\hat{M}'(x))_{i,j}=m_{i,j}
\frac{\theta_\ssq\bigl(\frac{\kappa_j}{\rho_i}\frac{x}{\xi}\bigr)} 
{\theta_\ssq\bigl(\frac{x}{\xi}\bigr)}\,,
\quad m_{i,j}\in\C\,.
\label{hypergeometricM}
\end{align}
The determinant relation
\begin{align}
m_{1,1}m_{2,2}\theta_{\ssq}\Bigl(\frac{\kappa_1}{\rho_1}\frac{x}{\xi}\Bigr) 
\theta_{\ssq}\Bigl(\frac{\kappa_2}{\rho_2}\frac{x}{\xi}\Bigr) 
-
m_{1,2}m_{2,1}\theta_{\ssq}\Bigl(\frac{\kappa_1}{\rho_2}\frac{x}{\xi}\Bigr) 
\theta_{\ssq}\Bigl(\frac{\kappa_2}{\rho_1}\frac{x}{\xi}\Bigr) 
=
K\theta_{\ssq}\Bigl(\frac{\eta x}{\xi}\Bigr)
\theta_{\ssq}\Bigl(\frac{x}{\xi}\Bigr)
\label{det-rel}
\end{align}
is equivalent to the following:
\begin{align}
&m_{1,1}m_{2,2}=K\frac{\theta_\ssq(\rho_2/\kappa_1)\theta_\ssq(\kappa_2/\rho_1)} 
{\theta_\ssq(\rho_2/\rho_1)\theta_\ssq(\kappa_2/\kappa_1)}\,,
\label{Mrel1}
\\
&m_{1,2}m_{2,1}=-K\frac{\kappa_2}{\rho_1}
\frac{\theta_\ssq(\rho_1/\kappa_1)\theta_\ssq(\rho_2/\kappa_2)} 
{\theta_\ssq(\rho_2/\rho_1)\theta_\ssq(\kappa_2/\kappa_1)}\,.
\label{Mrel2}
\end{align}
Indeed \eqref{Mrel1}, \eqref{Mrel2} 
are obtained by substituting into \eqref{det-rel} 
$x=\rho_2\xi/\kappa_1$ or $x=\rho_1\xi/\kappa_1$, respectively,
and using $\theta_\ssq(1)=0$. 
Conversely \eqref{Mrel1} and \eqref{Mrel2} imply \eqref{det-rel} 
due to the addition theorem for $\theta_\ssq(x)$.  

The above calculation includes 
the special case $N=0$, $p(x)=1$:
\begin{align}
Y(\ssq x)=A^{top}(x)Y(x)\,,\quad 
A^{top}(x)=\frac{x A_\infty-\xi A_0}{x-\xi}\,,
\label{Atop}
\end{align}
which can be reduced to a $q$-hypergeometric equation.
Let $Y^{' top}_0(x),Y^{' top}_\infty(x),\hat{M}^{' top}(x)$ be the 
corresponding modified canonical solutions
and the central part of the connection matrix, respectively. 
From the transformation formula of $q$-hypergeometric functions, see
e.g. (4.3.2) in \cite{GR}, we can find the constants $m^{top}_{i,j}$:
\begin{align*}
&m^{top}_{1,1}=\frac{(\rho_2/\kappa_1,\ssq\kappa_2/\rho_1;\ssq)_\infty} 
{(\rho_2/\rho_1,\ssq\kappa_2/\kappa_1;\ssq)_\infty} \frac{\kappa_1}{\rho_1}\,,
\quad
m^{top}_{1,2}=m^{top}_{1,1}\bigl|_{\kappa_1\leftrightarrow \kappa_2}\,,
\quad
m^{top}_{2,j}=m^{top}_{1,j}\bigl|_{\rho_1\leftrightarrow\rho_2}\,.
 \end{align*}

Assume further 
that\footnote{This is the condition for the system \eqref{Atop}
to be irreducible, see e.g. \cite{R}}
\begin{align}
\frac{\rho_i}{\kappa_j}\not\in \ssq^\Z\,,\quad i,j=1,2\,.
\label{irred}
\end{align}
Then $m_{ij}$ are all non-zero, and 
we can choose constant diagonal matrices $D_0,D_\infty$ such that
\begin{align*}
D_0\hat{M}'(x)D_\infty^{-1}=\hat{M}^{'top}(x)\,.  
\end{align*}
This can be rewritten as
\begin{align*}
\widehat{Y}'_0(x)\bigl(\widehat{Y}^{'top}_0(x)D_0\bigr) ^{-1}
=\widehat{Y}'_\infty(x)\bigl(\widehat{Y}^{'top}_\infty(x)D_\infty\bigr) ^{-1}\,.
\end{align*}
The left hand side is meromorphic on $\CP^1\backslash\{\infty\}$
and the right hand side is meromorphic on $\CP^1\backslash\{0\}$,
hence both sides are rational functions. Denote it by $R(x)$. 
Summarizing the argument we come to the conclusion

\begin{prop}
Notation being as above, 
assume the conditions \eqref{non-resonant} and \eqref{irred}. Then
there exist
 a rational matrix $R(x)\in GL_2(\C(x))$ which is regular at $x=0,\infty$, 
and $X_0,X_\infty\in GL_2(\C)$ satisfying
$[X_0,A_0]=0$, $[X_\infty,A_\infty]=0$,
such that the canonical solutions and the coefficient matrices are related 
as follows.
\begin{align*}
&Y_0(x)=R(x)Y_0^{top}(x)X_0\,,
\quad
 Y_\infty(x)=R(x)Y_\infty^{top}(x)X_\infty\,,
\\
&A(x)=R(\ssq x)A^{top}(x)R(x)^{-1}\,.
\end{align*}
\end{prop}

This is analogous to the result in \cite{ET} for hypergeometric 
differential equations. 
\bigskip

\section{``Unfolded'' Bethe equations}\label{sec:unfolded}

Let $\g$ be a complex simple Lie algebra with Cartan matrix 
$(C_{i,j})_{i,j\in I}$.
As is well known, 
with each finite dimensional representation $W$ of an untwisted 
quantum affine algebra $U_q\widehat{\g}$,
one can associate an XXZ type model. 
When $W$ is a tensor product of irreducible finite dimensional 
representations with highest $\ell$-weight
$\Psi_{i,k}(x)$, $i\in I$, $1\le k\le N$, 
the Bethe ansatz equations have the form 
\begin{align*}
p_i \prod_{j\in I}\prod_{s=1}^{m_j}
\frac{q^{B_{ij}}w^{(i)}_r-w^{(j)}_s } {w^{(i)}_r-q^{B_{ij}}w^{(j)}_s } 
=-\prod_{k=1}^{N}\Psi_{i,k}(w^{(i)}_r)\,,
\end{align*}
where $\{w^{(i_)}_r\}_{1\le r\le m_i}$ are the Bethe roots of ``color'' $i\in I$, $p_i$'s are twist parameters entering the transfer matrix, 
and 
$B_{ij}=d_iC_{ij}$ are the entries of the symmetrized Cartan matrix.

There is 
a similar but a different type of Bethe equations, see \cite{MV2}, in which the factor 
\begin{align*}
\frac{q^{B_{ij}} w^{(i)}_r-w^{(j)}_s } {w^{(i)}_r-q^{B_{ij}}w^{(j)}_s }  
\,,\quad \text{$i\neq j$}
\end{align*}
is replaced by
\begin{align*}
\frac{(q^{-1}w^{(i)}_r-w^{(j)}_s )^{-C_{ji}}} 
{(w^{(i)}_r-q^{-1}w^{(j)}_s )^{-C_{ji}}}\,,\quad \text{$i\neq j$}.
\end{align*}
The authors of \cite{FHR} called the latter Bethe equations ``folded type'' 
and proposed the corresponding quantum integrable models. 

Let us return to the Bethe equations \eqref{BE0}, \eqref{BE1}. 
In the special case $d=1$, i.e. $q_1=q_3=q^{-1}$,  these equations become
\begin{align*}
&p_0 
\prod_{k=1}^{l_0}\frac{q^2s_i-s_k} {s_i-q^2s_k} 
\prod_{j=1}^{l_1}\frac{(q^{-1}s_i-t_j)^2}{(s_i-q^{-1}t_j)^2}=-\Psi_0(s_i)\,,
\\
&p_1 
\prod_{l=1}^{l_1}\frac{q^2t_j-t_l} {t_j-q^2t_l} 
\prod_{i=1}^{l_0}\frac{(q^{-1}t_j-s_i)^2}{(t_j-q^{-1}s_i)^2}=-\Psi_1(t_j)\,.
\end{align*}
We recognize that 
they are ``folded type'' Bethe equations associated with the 
affine Cartan matrix
$C=
\begin{pmatrix}
  2& -2\\
-2 & 2\\
 \end{pmatrix}$.
This is to be compared with the Bethe equations of ``unfolded type'' 
obtained in another special case $d=q$ or $d=q^{-1}$,
\begin{align}
&p_0
\prod_{k=1}^{l_0}\frac{q^2s_i-s_k} {s_i-q^2s_k} 
\prod_{j=1}^{l_1}\frac{q^{-2}s_i-t_j}{s_i-q^{-2}t_j}
=-\Psi_0(s_i)\,,
\label{unfold0}\\
&p_1
\prod_{l=1}^{l_1}\frac{q^2t_j-t_l} {t_j-q^2t_l} 
\prod_{i=1}^{l_0}\frac{q^{-2}t_j-s_i}{t_j-q^{-2}s_i}
=-\Psi_1(t_i) \,.
\label{unfold1}
\end{align}

Here we are content to make one remark concerning 
a reformulation for the latter \eqref{unfold0}, \eqref{unfold1}
in terms of so-called discrete $\lambda$-opers.

Let $\ssq=q^2$. Define $P(x)$ by 
\begin{align*}
&p_0p_1\frac{P(\ssq x)}{P(x)}=\Psi_0(x)\Psi_1(x)\,. 
\end{align*}
In general $P(x)$ is not rational, even though $\Psi_i(x)$ are. 

Write $\Psi_1(x)=p_1\ssq^{l_0-l_1}{\bar c}(x)/{\bar a}(x)$ with 
polynomials $\bar{a}(x),\bar{c}(x)$.
Given polynomials
$y_0(x)=\prod_{i=1}^{l_0}(x-s_i)$, $y_1(x)=\prod_{i=1}^{l_1}(x-t_i)$ 
whose roots satisfy \eqref{unfold0}--\eqref{unfold1},
consider the following $q$-difference oper with the parameter $\lambda$
\begin{align*}
&L_\lambda=a(x)\sigma_\ssq^2 -b(x)\sigma_\ssq+c_\lambda(x)\,,
\qquad
c_\lambda(x)=c(x)(1-\lambda P(x))\,,
\end{align*}
where
\begin{align*}
&a(x)=\bar{a}(x)y_0(\ssq^{-1}x)\,,\\ 
&b(x)=a(x)\frac{y_1(\ssq x)}{y_1(x)}+c(x)\frac{y_1(\ssq^{-1}x)}{y_1(x)}\,,\\
&c(x)=\bar{c}(x)y_0(\ssq x)\,.
\end{align*}
Under the Bethe equations \eqref{unfold1},
$b(x)$ is a polynomial. 

\begin{prop}
The following are equivalent.

(i) The oper  $L_\lambda$ 
has an apparent singularity at $x=\ssq s_i$ for all $\lambda$.

(ii) The oper $L_0$ has an apparent singularity at $x=\ssq s_i$, 
and the zeroth Bethe equations \eqref{unfold0} hold.
\end{prop}
\begin{proof}
According to Proposition \ref{prop:2nd}, the condition for $\ssq s_i$ to be apparent is given by the vanishing of a single determinant
\begin{align*}
&\left|
\begin{matrix}
b(\ssq s_i) &  c_\lambda(\ssq s_i) & 0 \\
a(s_i )  &  b(s_i)  &  c_\lambda(s_i) \\
0 & a(\ssq^{-1}x) & b(\ssq^{-1}s_i)\\
\end{matrix}
\right| 
\\
&=
\left|
\begin{matrix}
b(\ssq s_i) &  c(\ssq s_i) & 0 \\
a(s_i )  &  b(s_i)  &  c(s_i) \\
0 & a(\ssq^{-1}x) & b(\ssq^{-1}s_i)\\
\end{matrix}
\right| 
-\lambda \Bigl(
P(s_i)c(s_i)b(\ssq s_i)a(\ssq^{-1} s_i)
+P(\ssq s_i)c(\ssq s_i)b(\ssq^{-1} s_i)a(s_i)
\Bigr)\,.
\end{align*}
The coefficient of $\lambda$ vanishes if and only if
\begin{align}
-1&=\frac{P(s_i)}{P(\ssq s_i)}\frac{a(\ssq^{-1}s_i)}{a(s_i)}
\frac{b(\ssq s_i)}{b(\ssq^{-1} s_i)}\frac{c(s_i)}{c(\ssq s_i)}\,.
\label{lambda-vanish}
\end{align}
Noting that 
\begin{align*}
&b(\ssq s_i)=c(\ssq s_i)\frac{y_1(s_i)}{y_1(\ssq s_i)}\,, \quad
b(\ssq^{-1} s_i)=a(\ssq^{-1} s_i)\frac{y_1(s_i)}{y_1(\ssq^{-1} s_i)}\,, 
\end{align*}
and using the Bethe equations \eqref{unfold1},
we find that \eqref{lambda-vanish} becomes
\begin{align*}
-1=&p_0\ssq^{-l_0+l_1}\frac{1}{\Psi_0(s_i)}
\frac{y_0(\ssq s_i)} {y_0(\ssq^{-1} s_i)} 
\frac{y_1(\ssq^{-1} s_i)}{y_1(\ssq s_i)}\,,
\end{align*}
which is nothing but \eqref{unfold0}.
\end{proof}

So far we have not been able to rewrite the conditions 
\eqref{apparent-relax} and \eqref{BE0-relax} into the ``$\lambda$-form''
for the $q$-oper \eqref{oper-relax} with apparent singularities
for general $q,d$. 

\vskip 1cm

{\bf Acknowledgments.\ }
The study of BF has been funded within the framework of the HSE University Basic Research Program. 
MJ is partially supported by JSPS KAKENHI Grant Numbers 19K03509, 20K03568.
EM is partially supported by Simons Foundation grant number \#709444.

BF thanks Jerusalem University for hospitality. 
BF and EM thank Rikkyo University, where a part of this work was done, 
for hospitality.

\end{document}